\documentclass[11pt]{amsart}
\usepackage[T1]{fontenc}
\usepackage[utf8]{inputenc}
\usepackage[english]{babel}
\usepackage{lmodern}
\usepackage{microtype}
\usepackage{amsmath,amssymb,amsthm,mathtools}
\usepackage{enumitem}
\usepackage[a4paper,margin=28mm]{geometry}

\allowdisplaybreaks
\newtheorem{theorem}{Theorem}[section]
\newtheorem{proposition}[theorem]{Proposition}
\newtheorem{lemma}[theorem]{Lemma}
\newtheorem{corollary}[theorem]{Corollary}

\theoremstyle{definition}
\newtheorem{definition}[theorem]{Definition}
\newtheorem{example}[theorem]{Example}
\theoremstyle{remark}
\newtheorem{remark}[theorem]{Remark}

\newcommand{\tto}{\rightrightarrows}
\newcommand{\bx}{\bar x}
\newcommand{\by}{\bar y}
\newcommand{\bv}{\bar v}
\newcommand{\Lin}{\mathcal L}
\newcommand{\R}{\mathbb R}
\newcommand{\CB}{\mathcal{CB}}
\newcommand{\gph}{\mathop{\rm gph}\nolimits}
\newcommand{\domm}{\mathop{\rm dom}\nolimits}
\newcommand{\rge}{\mathop{\rm rge}\nolimits}
\newcommand{\dist}{\mathop{\rm dist}\nolimits}

\def\ball{{I\kern -.35em B}}

\title[Higher-order semiregularity via acyclic ranges]%
{Anisotropic Higher-Order Semiregularity of Degenerate Generalized Equations}
\author{Tom\'a\v{s} Roubal}
\address{Institute of Information Theory and Automation, Czech Academy of
	Sciences, Prague, Czech Republic}
\email{roubal@utia.cas.cz}
\urladdr{https://orcid.org/0000-0002-6137-1046}
\date{}

\begin{document}
	
	\begin{abstract}
		
		We give a self-contained triangular formulation of anisotropic higher-order
		covering and inverse estimates for smooth mappings and generalized equations.
		For a $C^p$ mapping into a finite-dimensional target, an indexed target
		decomposition is fixed, the corresponding triangular factor condition is
		imposed, and a direction satisfying the associated triangular $p$-kernel
		condition is chosen. Surjectivity of the resulting triangular $p$-factor
		operator yields a fixed-scale inclusion in which each nonzero block $Y_i$ is
		covered at order $t^i$, and hence an inverse estimate with exponent $1/i$ on
		that block. Zero blocks are allowed, and $p$ denotes the highest active grade.
		
		The smooth result is a derivative-level graded formulation: it verifies the
		model directly from the $C^p$ derivatives and records the blockwise
		fixed-scale inclusion needed for the set-valued analysis.
		
		For generalized equations, an auxiliary exact-model theorem isolates the
		range-transfer mechanism with acyclic correction fibres. The main sufficient
		criterion uses a closed convex-process inner approximation whose sum with
		the smooth triangular factor operator is surjective, allowing the set-valued
		term to supply missing directions.
		
	\end{abstract}
	
	\maketitle
	
	\section{Introduction}
	
	At a regular point, surjectivity of the first derivative gives a linear
	covering and a Lipschitz inverse estimate. At a degenerate point, higher
	derivatives may recover missing target directions at different rates. We
	study the resulting fixed-base semiregularity. Given an indexed
	decomposition
	$$
	Y=Y_1\oplus\cdots\oplus Y_p
	$$
	with projections $P_i$, set
	$$
	q_p(v):=\max_{1\le i\le p}\|P_i v\|^{1/i}.
	$$
	For a mapping $\mathcal M:X\tto Y$ and
	$(\bx,\by)\in\gph\mathcal M$, the conclusion obtained below has the form
	\begin{equation}
		\dist\bigl(\bx,\mathcal M^{-1}(y)\bigr)
		\le Mq_p(y-\by)
		\qquad (y\ \hbox{near }\by).
		\label{eq:intro-fixed-base-semiregularity}
	\end{equation}
	Thus the source point in the distance on the left-hand side is always the
	fixed point $\bx$.
	
	This quantifier structure distinguishes \eqref{eq:intro-fixed-base-semiregularity}
	from H\"older metric regularity. For $q\in(0,1]$, metric $q$-regularity
	around $(\bx,\by)$ requires neighborhoods $U$ and $V$ and a constant
	$\kappa>0$ such that
	\begin{equation}
		\dist\bigl(x,\mathcal M^{-1}(y)\bigr)
		\le
		\kappa\,\dist\bigl(y,\mathcal M(x)\bigr)^q
		\quad\text{for every}\quad (x,y)\in U\times V.
		\label{eq:intro-metric-q-regularity}
	\end{equation}
	In particular, the point $x$ varies in
	\eqref{eq:intro-metric-q-regularity}. H\"older metric subregularity fixes
	$y=\by$ instead, whereas semiregularity fixes $x=\bx$; see
	\cite{CibulkaFabianKruger,FrankowskaQuincampoix,
		KrugerHolderSubregularity}. Consequently, none of the results below asserts
	metric $q$-regularity around the graph point. If $Y_p$ is the only active
	block, then $q_p(v)=\|v\|^{1/p}$ and
	\eqref{eq:intro-fixed-base-semiregularity} is a H\"older semiregularity
	estimate of order $1/p$. With several active blocks, the estimate retains
	the sharper rate $1/i$ on $Y_i$ and implies, after localization, a scalar
	H\"older semiregularity estimate of order $1/p$. The two comparisons should
	not be conflated: metric $q$-regularity has stronger base-point uniformity,
	whereas the anisotropic error function $q_p$ contains finer componentwise information than a
	single worst-order H\"older power.
	
	\subsection*{Position with respect to higher-order inverse theorems}
	
	The higher-order inverse theorems of Frankowska cited below are formulated through
	higher-order variations of single- and set-valued mappings on complete
	metric spaces; see, in particular,
	\cite[Theorems~4.1 and~5.7]{FrankowskaHighOrder} and
	\cite[Theorem~2.1 and Section~7]{FrankowskaSomeInverse}. Their
	principal openness and inverse conclusions are uniform with respect to
	nearby graph points: the statements quantify over
	$(x,v)\in\gph\mathcal M$ near $(\bx,\by)$ and over nearby target
	perturbations. The resulting inverse estimates use a single H\"older order,
	namely $1/k$ in the $k$th-order statements. Frankowska and Quincampoix
	subsequently study this
	property explicitly as H\"older metric regularity, also called metric
	$q$-regularity \cite{FrankowskaQuincampoix}; the square-root
	metric-regularity results in
	\cite{ArutyunovKaramzinSquareRoot} have the same graph-uniform character
	in the second-order setting. The present conclusions concern a different
	regularity property: they give the fixed-base estimate
	\eqref{eq:intro-fixed-base-semiregularity} with an anisotropic target gauge.
	Accordingly, the present results neither strengthen nor generalize those
	metric-regularity theorems; they trade neighborhood uniformity in the
	source point for a blockwise fixed-base modulus suited to the subsequent
	generalized-equation transfer.
	
	The relation with Sussmann's homogeneous framework also requires a
	qualification. In \cite[Definitions~6, 8, and~11 and
	Theorem~12]{SussmannHighOrderOpen}, the dilation group acts on the source,
	$\nu$ is a homogeneous source pseudonorm, and the model satisfies
	$$
	G(\delta_t\xi)=tG(\xi)
	$$
	with ordinary scalar multiplication in the target; the approximation error
	is $o(\nu(\xi))$. Here, by contrast, the graded Taylor model satisfies, for
	every $s>0$,
	$$
	\mathcal T_p(sz)=sD_s\mathcal T_p(z)
	\quad\text{and}\quad
	q_p(sD_sy)=sq_p(y).
	$$
	Thus the grading acts in the target and $q_p$ measures target errors.
	Under the assumed surjectivity of $\Psi_p(h)$, the identities
	$\mathcal T_p(h)=0$ and $D\mathcal T_p(h)=\Psi_p(h)$ make $h$ a regular
	zero of the graded model. They do not, however, supply Sussmann's
	source-dilation homogeneity. His theorem is therefore not a direct
	consequence of the hypotheses stated here on the same source space. Compatible
	quasi-homogeneous or augmented-space reformulations may cover special
	cases, and no contrary claim is made. The role of
	Theorem~\ref{thm:p-factor-covering} is to verify directly from the
	$C^p$ derivatives the target-graded fixed-scale family. More precisely, for
	every finite-dimensional subspace $E\subset X$ satisfying
	$\Psi_p(h)(E)=Y$, there exist $c,\rho,r>0$ such that
	\begin{equation}
		\label{eq:intro-openness}
		f(\bx)+D_t(ct\ball_Y^\Sigma)
		\subset f(\bx+th+\rho t\ball_E)
		\quad\text{for every}\quad t\in(0,r],
	\end{equation}
	which is then used in the set-valued range-transfer argument. On the correction
	space $E$, this conclusion can be viewed through the rescaled family
	$$
	\Phi_t(z):=\frac{1}{t}D_t^{-1}
	\bigl(f(\bx+t(h+z))-f(\bx)\bigr),
	\qquad z\in E.
	$$
	Its derivative is
	$$
	D\Phi_t(z)=D_t^{-1}f'(\bx+t(h+z))\vert_E.
	$$
	Lemma~\ref{lem:p-factor-scaled-error} shows that, for every
	$\varepsilon>0$, the operators $D\Phi_t(z)$ are uniformly
	$\varepsilon$-close to $\Psi_p(h)\vert_E$ when $z$ lies in a sufficiently
	small neighborhood of $0$ and $t$ is sufficiently small, while
	$\Phi_t(0)\to0$. Thus the smooth theorem is a uniform first-order openness
	argument for this rescaled family; the higher-order content lies in deriving
	that approximation from the triangular $C^p$ data and retaining the
	blockwise target scales.
	
	For a standard recursive target decomposition, the operator
	$\Psi_p(h)$ used here agrees, up to a fixed invertible diagonal
	normalization, with the standard $p$-factor operator. Surjectivity for one
	chosen $h$ is therefore the established notion of $p$-regularity along
	$h$ \cite{BrezhnevaTretyakovMarsden,TretyakovMarsden,
		BednarczukEtAlSurvey}. Related singular tangency results are developed in
	\cite{BednarczukTretyakovNonlinearity,PrusinskaTretyakovTangent}. It is
	weaker than $p$-regularity at the reference
	point, which requires the condition along every direction in the exact
	$p$-kernel, and weaker than strong $p$-regularity, which imposes a uniform
	right-inverse bound over approximate kernel directions. For a general
	triangular decomposition used below, no identification with the standard
	recursive construction is asserted.
	A closely related selected-direction result was obtained by Prusińska and Tret'yakov \cite[Theorem~3]{PrusinskaTretyakov2011}. For a $C^{p+1}$ mapping between Banach spaces, they assume that a direction $h$ belongs to the corresponding nonzero $p$-kernel and that the mapping is $p$-regular along $h$. Under additional quantitative smallness conditions involving the residual at the reference point, they prove the existence of a zero of the form
	$$
	x^*=x_0+\omega h+\bar x(\omega),
	\qquad
	\|\bar x(\omega)\|\leq \frac{\omega}{2},
	$$
	with an $o(\omega)$ refinement under a stronger residual scaling. Their theorem is an existence result for a single equation. By contrast, Theorem~5.3 below gives, under triangular $C^p$ assumptions, a uniform moving-tube range inclusion for all anisotropically scaled target perturbations, retains the individual block scales $t^i$, and does not require a derivative of order $p+1$. Conversely, the result in \cite{PrusinskaTretyakov2011} is formulated for general Banach target spaces and directly incorporates a nonzero residual through quantitative conditions. Thus neither result directly subsumes the other; the contribution here is the target-graded covering conclusion and its subsequent use in the set-valued range-transfer argument.
	
	In the standard second-order case,
	the smooth estimate specializes to the result of Izmailov and Solodov
	\cite[Theorem~4.1]{IzmailovSolodov}; no new second-order estimate is
	claimed in that case.
	
	There is a further overlap with the theory of $\lambda$-truncations.
	After restriction to a finite-dimensional correction space and use of
	unit input weights and output degree $i$ on $Y_i$, the graded polynomial
	$\mathcal T_p$ is the relevant truncation. If every coordinate remainder
	admits a finite monomial majorant whose weighted degrees are strictly
	larger than the degree of the corresponding output coordinate, the
	componentwise inverse estimate follows from
	\cite[Theorem~3.2]{ArutyunovZhukovskiySbornik2025}; see also
	\cite{ArutyunovZhukovskiyMathNotes2025}. Hence no novelty is claimed in
	this strict weighted power-gap subclass. The distinction used here is
	that, under the triangular factor condition, $C^p$ regularity yields in the
	top block only
	$$
	P_p\bigl(f(\bx+z)-f(\bx)-\mathcal T_p(z)\bigr)=o(\|z\|^p)
	$$
	in general, with no estimate $O(\|z\|^{p+\gamma})$ for any
	$\gamma>0$. The proof below uses precisely this little-$o$ information,
	and Remark~\ref{rem:logarithmic-remainder} shows that it cannot always be
	replaced by a strict weighted power gap. This observation distinguishes the present
	hypotheses from the cited $\lambda$-truncation theorem; by itself it is
	not a priority claim relative to the broader homogeneous higher-order
	inverse literature.
	
	At third order, Boarotto, Monti, and Palmurella prove qualitative
	openness for smooth maps from a Banach space to a finite-dimensional
	manifold using intrinsic second- and third-order differentials
	\cite[Theorem~1.1]{BoarottoMontiPalmurella}. Their arbitrary-corank
	condition uses an isotropic second-order direction and a possibly
	different regular zero of the intrinsic third differential, and is not a
	restatement of the single-direction triangular condition used here. Their
	endpoint-map analysis computes the intrinsic third differential and
	derives necessary conditions for singular sub-Riemannian minimizers
	\cite[Theorem~1.2 and Proposition~3.5]{BoarottoMontiPalmurella}. The
	terminal results below are complementary: they apply arbitrary-order
	target-graded estimates to generalized endpoint inclusions and retain an
	explicit fixed-base anisotropic inverse modulus.
	
	Thus the contribution is not the exponent $1/p$, the idea of regularity
	along a selected direction, or a new abstract higher-order open mapping
	principle. It is the derivative-level verification, under the stated
	triangular $C^p$ hypotheses, of an explicit blockwise fixed-scale
	inclusion, together with its transfer to generalized equations through
	acyclic correction fibres and a convex-process mechanism by which the
	set-valued term may supply directions missing from the smooth factor
	operator.
	
	The ordered target decomposition is part of the hypotheses rather than a
	canonical object determined by $f$. Changing the decomposition, the
	projections, or the order labels may change the triangular conditions,
	factor surjectivity, and the resulting componentwise estimate. The convention
	concerning zero blocks is stated precisely in
	Remark~\ref{rem:grading-dependence}.
	
	The first main result concerns smooth mappings. For a $C^p$
	mapping with finite-dimensional target satisfying the triangular factor
	condition relative to a chosen decomposition, surjectivity of $\Psi_p(h)$
	together with $\Psi_p(h)h=0$ gives \eqref{eq:intro-openness}
	and consequently
	$$
	\dist\bigl(\bx,f^{-1}(y)\bigr)
	\le M\max_{1\le i\le p}
	\|P_i(y-f(\bx))\|^{1/i}.
	$$
	This moving-tube inclusion is the fixed-scale input for the set-valued
	analysis below.

	For the proofs, we record a compact acyclic range-transfer statement in the
	form needed later.  It is an immediate consequence of the
	Eilenberg--Montgomery fixed-point theorem \cite{EilenbergMontgomery}; no
	independent fixed-point result or general advance in range-transfer theory
	is claimed.  Its role is auxiliary: it separates model covering,
	approximation error, and correction-fibre topology.  The convex-domain,
	convex-fibre case may be compared with
	\cite[Theorem~2.1]{CibulkaRoubalSetValued}.  A separate remark records the
	Hausdorff-contractive alternative.
	
	We then turn to the generalized equation
	$$
	y\in f(x)+F(x),
	$$
	and distinguish an auxiliary transfer layer from a structural verification
	criterion. Section~\ref{sec:partial-p-factor} assumes that the exact partial
	triangular model has a fixed-scale covering and acyclic localized correction
	fibres, and transfers these properties to $f+F$ under anisotropically
	negligible approximation and centring errors. This is a transfer template,
	not a directly verifiable regularity criterion for $F$.
	
	The main generalized-equation result is the convex-process criterion of
	Section~\ref{sec:process-criterion}. A finite-dimensional closed convex
	process is embedded into the set-valued term through an anisotropically
	small continuous remainder. Surjectivity of its sum with the triangular
	factor operator yields the componentwise covering and inverse estimate,
	allowing $F$ to supply target directions missing from the smooth factor
	operator. A finite-dimensional hybrid example shows that the exact-model transfer
	theorem can still apply when no compatible zero-remainder exact
	convex-process inner model exists on the chosen correction space. These
	fixed-base results complement
	existing $p$-regularity and coderivative approaches to singular generalized
	equations
	\cite{BednarczukPrusinskaTretyakov,ChuongPositiveOrderGE}.
	
	Sections~\ref{sec:preliminaries}--\ref{sec:acyclic-range} introduce the
	anisotropic notation and the range-transfer tool.
	Section~\ref{sec:perturbation} records an ancillary fixed-base perturbation
	result.
	Sections~\ref{sec:p-factor}--\ref{sec:process-criterion} develop the smooth
	covering theorem, the exact-model transfer, and the convex-process criterion.
	Section~\ref{sec:examples} contains selected-direction, set-valued, and
	terminal-control examples.
	
	\section{Topological and anisotropic preliminaries}
	\label{sec:preliminaries}
	
	\subsection{Acyclic spaces and fixed-point tools}
	
	Given a metric space $(X,d)$, the closed and open balls centred at
	$x\in X$ with radius $r>0$ are denoted by $\ball_X[x,r]$ and
	$\ball_X(x,r)$, respectively.  The family of all nonempty closed bounded
	subsets of $X$ is denoted by $\CB(X)$ and the Hausdorff distance on
	$\CB(X)$ by $d_H$.
	
	For a set-valued mapping $F:X\tto Y$, its graph, domain, range, and inverse
	are denoted by $\gph F$, $\domm F$, $\rge F$, and $F^{-1}$, respectively.
	Thus
	$$
	F^{-1}(y):=\{x\in X:y\in F(x)\}.
	$$
	If $\Omega\subset X$, then $F\vert_{\Omega}$ denotes the restriction
	satisfying
	$$
	\gph(F\vert_{\Omega})=\gph F\cap(\Omega\times Y).
	$$
	For locally convex spaces $X$ and $Y$, the symbol $\Lin(X,Y)$ denotes the
	space of continuous linear mappings from $X$ to $Y$.
	We use the conventions $\dist(x,\varnothing):=+\infty$,
	$\sup\varnothing:=0$, and $\inf\varnothing:=+\infty$.
	
	A nonempty compact topological space $K$ is called \emph{acyclic} if
	$$
	\widetilde{\check H}_j(K;\mathbb Q)=0
	\quad\text{for every }j\ge 0,
	$$
	where $\widetilde{\check H}_j$ denotes reduced \v{C}ech homology with
	rational coefficients; see \cite[Chapter~I]{Gorniewicz} and
	\cite[Chapters~IX--X]{EilenbergSteenrod}.
	
	A metrizable space $K$ is called an absolute neighborhood retract
	(\emph{ANR}) if, whenever $K$ is embedded as a closed subset of a
	metrizable space $Z$, there exist an open neighborhood $U$ of $K$ in
	$Z$ and a continuous retraction
	$r:U\longrightarrow K.$
	
	A \emph{Fr\'echet space} is a Hausdorff complete metrizable locally convex
	topological vector space.
	
	Every nonempty compact convex subset of a Fr\'echet space is a compact
	metrizable absolute retract. Consequently, it is an acyclic ANR.
	
	The following form of the Eilenberg--Montgomery fixed-point theorem
	replaces convexity of the values by acyclicity and allows the domain
	to be any compact metrizable acyclic ANR. We use its closed-graph
	formulation. Indeed, for a set-valued mapping between compact Hausdorff
	spaces, closedness of the graph implies upper semicontinuity. To prepare the
	passage from smooth mappings to generalized equations, we
	first record an auxiliary acyclic range-transfer result.
	
	\begin{theorem}[Eilenberg--Montgomery]
		\label{thm:EM}
		Let $\Omega$ be a nonempty compact metrizable acyclic absolute
		neighborhood retract. Consider a set-valued mapping
		$\Phi:\Omega\tto\Omega$ with a closed graph and acyclic values.
		Then $\Phi$ has a fixed point.
	\end{theorem}
	
	This is the specialization of
	\cite[Theorem~1, p.~215]{EilenbergMontgomery} obtained by taking the
	auxiliary compact metric space equal to $\Omega$ and the retraction equal
	to the identity. In the \v{C}ech-homology terminology used here, see also
	\cite[Definition~32.1 and Corollary~32.12, pp.~159, 163]{Gorniewicz};
	for further background, see \cite{GranasDugundji}.
	
	\subsection{Anisotropic error sets and dilations}
	
	\begin{definition}
		\label{def:anisotropic-structure}
		Let $p\ge1$ be an integer, let $Y$ be a nonzero normed space, and let
		$$
		Y:=Y_1\oplus\cdots\oplus Y_p
		$$
		be a topological direct-sum decomposition.  The associated continuous
		projections are denoted by $P_i:Y\longrightarrow Y_i$.  We use the
		reduced-index convention $Y_p\ne\{0\}$; initial and intermediate zero
		blocks are allowed.  Set
		$$
		\mathcal I:=\{i\in\{1,\ldots,p\}:Y_i\ne\{0\}\}.
		$$
		We equip $Y$ with
		the equivalent sum norm
		\begin{align*}
			\|y\|_{\Sigma}:=\sum_{i=1}^p\|P_i y\|.
		\end{align*}
		Its closed unit ball is denoted by $\ball_Y^{\Sigma}$.  For a normed space
		$E$ and a linear operator $T\in\Lin(E,Y)$, the corresponding operator norm is denoted by
		$\|T\|_{\Sigma}$, that is,
		$$
		\|T\|_{\Sigma}:=\sup_{\|e\|\le1}\|Te\|_{\Sigma}.
		$$
		The associated anisotropic error function is
		$$
		q_p(y):=\max_{i\in\mathcal I}\|P_i y\|^{1/i},
		$$
		and its closed sublevel sets are denoted by
		$$
		Q_p(s):=\{y\in Y:q_p(y)\le s\}
		\quad\text{for}\quad s\ge0.
		$$
		
		Equivalently,
		$$
		Q_p(s)=\{y\in Y:\|P_i y\|\le s^i,\ i=1,\ldots,p\}.
		$$
		The function $q_p$ is continuous, being a finite maximum of continuous
		functions, and it is subadditive. Indeed, for $y,z\in Y$ and every $i$,
		$$
		\begin{aligned}
			\|P_i(y+z)\|^{1/i}
			\le\bigl(\|P_i y\|+\|P_i z\|\bigr)^{1/i}
			\le\|P_i y\|^{1/i}+\|P_i z\|^{1/i}
			\le q_p(y)+q_p(z).
		\end{aligned}
		$$
		Consequently,
		$$
		Q_p(a)+Q_p(b)\subset Q_p(a+b)
		\quad\text{for all}\quad a,b\ge0.
		$$
		Each $Q_p(s)$ is closed and convex, and it is compact whenever $Y$ is
		finite-dimensional.    Both $q_p$ and the dilations below are
		relative to the chosen direct-sum decomposition.
		
		For every $t>0$, the \emph{normalized anisotropic scaling} is the bounded linear
		isomorphism
		$$
		D_t:=\sum_{i=1}^p t^{i-1}P_i,
		\quad\text{then}\quad
		D_t^{-1}=\sum_{i=1}^p t^{1-i}P_i.
		$$
		Thus $tD_t$ is the \emph{full order-$t$ anisotropic dilation}: for every $y\in Y$,
		$$
		q_p(tD_t y)=tq_p(y).
		$$
		Moreover,
		$$
		D_t(t\ball_Y^\Sigma)
		\subset Q_p(t)
		\subset D_t(pt\ball_Y^\Sigma).
		$$
		For the first inclusion, write $y=D_t(tu)$ with
		$\|u\|_\Sigma\le1$ and observe that
		$\|P_i y\|=t^i\|P_i u\|\le t^i$. Conversely, if $y\in Q_p(t)$, then
		$$
		\|D_t^{-1}y\|_\Sigma
		=\sum_{i=1}^pt^{1-i}\|P_i y\|
		\le pt,
		$$
		which proves the second inclusion.
		Thus $Q_p(t)$ and $D_t(t\ball_Y^\Sigma)$ describe equivalent anisotropic
		neighborhoods, but they are not identical in general.  For a set $B\subset Y$
		and a set-valued mapping $\mathcal H:E\tto Y$, we use the conventions
		$$
		D_tB:=\{D_t y:y\in B\}
		\quad\text{and}\quad
		(D_t\mathcal H)(e):=D_t\bigl(\mathcal H(e)\bigr).
		$$
	\end{definition}
	\begin{remark}[Equivalence of anisotropic openness and semiregularity]
		\label{rem:anisotropic-covering-inverse-equivalence}
		Let $X$ be a normed space, let $F:X\tto Y$, and let
		$(\bx,\by)\in\gph F$.  The following two properties are equivalent:
		\begin{enumerate}[label={\rm (\roman*)}]
			\item there are $M>0$ and a neighborhood $V$ of $\by$ such that
			\begin{equation}
				\dist\bigl(\bx,F^{-1}(y)\bigr)
				\le Mq_p(y-\by)
				\quad\text{for every}\quad y\in V;
				\label{eq:anisotropic-semiregularity}
			\end{equation}
			\item there are $a,r>0$ such that
			\begin{equation}
				\by+Q_p(at)
				\subset F\bigl(\ball_X[\bx,t]\bigr)
				\quad\text{for every}\quad t\in(0,r].
				\label{eq:anisotropic-openness}
			\end{equation}
		\end{enumerate}
		Property~{\rm (i)} is called anisotropic semiregularity relative to the
		chosen decomposition, and property~{\rm (ii)} is called anisotropic openness
		relative to that decomposition. More precisely,
		\eqref{eq:anisotropic-semiregularity} with a constant
		$M$ implies \eqref{eq:anisotropic-openness} with every
		$a\in(0,1/M)$ after decreasing $r$.  Conversely,
		\eqref{eq:anisotropic-openness} with a constant $a$ implies
		\eqref{eq:anisotropic-semiregularity} locally with the constant $1/a$.
		
		Indeed, suppose first that \eqref{eq:anisotropic-semiregularity} holds
		and fix $a\in(0,1/M)$.  Choose $r>0$ so small that
		$$
		\by+Q_p(ar)\subset V.
		$$
		If $t\in(0,r]$ and $w\in\by+Q_p(at)$, then
		$$
		\dist\bigl(\bx,F^{-1}(w)\bigr)
		\le Mq_p(w-\by)
		\le Mat<t.
		$$
		Hence $F^{-1}(w)$ meets $\ball_X[\bx,t]$, which proves
		\eqref{eq:anisotropic-openness}.
		
		Conversely, suppose that \eqref{eq:anisotropic-openness} holds and let
		$y\ne\by$ satisfy $q_p(y-\by)<ar$.  Set
		$$
		t:=\frac{q_p(y-\by)}{a}.
		$$
		Then $t\in(0,r)$ and $y\in\by+Q_p(at)$.  Therefore
		$$
		\dist\bigl(\bx,F^{-1}(y)\bigr)
		\le t
		=\frac1a q_p(y-\by).
		$$
		The estimate is immediate for $y=\by$.
	\end{remark}
	
	For any normed space $E$, we write $\ball_E:=\ball_E[0,1]$.
	
	\begin{corollary}
		\label{cor:moving-tube-inverse}
		Let $X$ be a normed space, let $\mathcal U\subset X$, and let
		$\mathcal M:\mathcal U\tto Y$. Fix
		$(\bx,\by)\in\gph\mathcal M$, $h\in X$, and a linear subspace $E\subset X$.
		Suppose that there are $c,\rho,r>0$ such that, for every $t\in(0,r]$,
		$$
		\bx+th+\rho t\ball_E\subset\mathcal U
		$$
		and
		\begin{equation}
			\by+D_t(ct\ball_Y^\Sigma)
			\subset\mathcal M(\bx+th+\rho t\ball_E).
			\label{eq:moving-tube-cover}
		\end{equation}
		Then there are $M_0>0$ and a neighborhood $V$ of $\by$ such that
		$$
		\dist\bigl(\bx,\mathcal M^{-1}(y)\bigr)
		\le M_0q_p(y-\by)
		\quad\text{for every}\quad y\in V.
		$$
		If a set $\mathcal A\subset\mathcal U$ contains $\bx$ and
		$\bx+th+\rho t\ball_E$ for every $t\in(0,r]$, then the same conclusion holds with
		$\mathcal M^{-1}(y)$ replaced by
		$\mathcal M^{-1}(y)\cap\mathcal A$.
	\end{corollary}
	
	\begin{proof}
		Choose $K>1$ so that
		$$
		\sum_{i\in\mathcal I}K^{-i}\le c
		$$
		and let
		$$
		V:=\{y\in Y:q_p(y-\by)<r/K\}.
		$$
		Fix $y\in V$, let $d:=y-\by$, and first suppose that $d\ne0$. With
		$t:=Kq_p(d)$, we have $t\in(0,r)$ and
		$$
		\begin{aligned}
			\|D_t^{-1}d\|_\Sigma
			&=\sum_{i\in\mathcal I}t^{1-i}\|P_i d\|\\
			&\le t\sum_{i\in\mathcal I}K^{-i}
			\le ct.
		\end{aligned}
		$$
		Thus \eqref{eq:moving-tube-cover} supplies a preimage of $y$ in
		$\bx+th+\rho t\ball_E$, and hence
		$$
		\dist\bigl(\bx,\mathcal M^{-1}(y)\bigr)
		\le K(\|h\|+\rho)q_p(d).
		$$
		The assertion for $d=0$ follows from $(\bx,\by)\in\gph\mathcal M$.
		The same argument proves the restricted conclusion when all the tubes and
		$\bx$ lie in $\mathcal A$.
	\end{proof}
	
	\begin{remark}[Fixed-base versus metric $q$-regularity]
		The word \emph{semiregularity} is essential here. Metric $q$-regularity
		allows the point $x$ in
		$$
		\dist(x,F^{-1}(y))
		\le\kappa\,\dist(y,F(x))^q
		$$
		to vary near $\bx$, whereas \eqref{eq:anisotropic-semiregularity} keeps
		$x=\bx$ fixed and varies only $y$. These properties are genuinely
		different even in the scalar linear-rate case. For example,
		$$
		F(x):=\{0,x\},\qquad x\in\mathbb R,
		$$
		is semiregular at $(0,0)$ because
		$\dist(0,F^{-1}(y))=|y|$ for $y\ne0$, but it is not metrically
		$q$-regular at $(0,0)$ for any $q>0$: fix a small $x\ne0$ and let
		$y\to0$, $y\ne0$. Then
		$\dist(x,F^{-1}(y))=|x-y|$, while
		$\dist(y,F(x))=|y|$ for all sufficiently small $y$.
	\end{remark}
	
	\begin{remark}[Dependence on the indexed decomposition and zero blocks]
		\label{rem:grading-dependence}
		The ordered family $(Y_1,\ldots,Y_p)$ and its projections are part of the
		data.  Accordingly, $q_p$, $D_t$, the later triangular conditions and
		factor operators, and all componentwise estimates are relative to this
		fixed choice.  If the decomposition, the projections, or the order labels
		are changed while the mapping is kept fixed, the hypotheses must be
		verified again; no invariance under an arbitrary regrading is asserted.
		A grade-preserving target isomorphism conjugates the algebraic construction,
		while the quantitative constants change according to the norms of its block
		restrictions and their inverses.
		
		If $Y_i=\{0\}$, then $P_i=0$, and all terms and conditions belonging to
		that block are vacuous.  Nevertheless, an initial or intermediate zero
		block is retained because its position records a missing effective order:
		deleting it and renumbering subsequent blocks would change the powers in
		$q_p$ and $D_t$ and the derivative orders in the factor operator.  Only
		trailing zero blocks may be discarded without changing the indexed
		structure.  Thus, under the reduced-index convention, $p=\max\mathcal I$
		is the highest active order.  Maxima and sums written below over
		$i=1,\ldots,p$ are unchanged if they are restricted to $\mathcal I$.
		A zero target block should not be confused with a nonzero block on which
		the later factor term $L_i(h)$ vanishes: the former is vacuous, whereas
		the latter remains a genuine target component and prevents surjectivity of
		the smooth factor operator.  In the generalized-equation criteria, such a
		missing component may instead be supplied by the set-valued term.
	\end{remark}
	
	The two ingredients introduced above play different roles: acyclicity
	supplies a fixed point for the correction mapping, whereas $q_p$ and
	$D_t$ encode the target scales. The next section records a convenient
	consequence of the Eilenberg--Montgomery theorem in the form used later.
	The abstract statement itself does not depend on the chosen indexed
	decomposition; that dependence enters only in applications through the
	specialization $S=D_t$ and the associated sets $Q_p(t)$.
	
	\section{Ranges of set-valued mappings}
	\label{sec:acyclic-range}
	
	The statements in this section are consequences of the
	Eilenberg--Montgomery fixed-point theorem and are recorded as auxiliary
	tools for the later arguments.
	
	\begin{proposition}
		\label{thm:ambient-range}
		Let $X$ be a Hausdorff topological space, let $Y$ be a Hausdorff
		topological vector space, and let $y\in Y$. Let $\Omega\subset X$ be
		a nonempty compact metrizable acyclic ANR. Consider a set-valued
		mapping $G:X\tto Y$ such that $G|_{\Omega}$ has a closed graph, and
		a continuous mapping $h:\Omega\longrightarrow Y$. Assume that
		$$
		G^{-1}(y-h(u))\cap\Omega
		$$
		is acyclic for every $u\in\Omega$. Then there exists $x\in\Omega$
		such that
		$$
		y\in h(x)+G(x).
		$$
	\end{proposition}
	
	\begin{proof}
		Consider the mapping $\Phi:\Omega\tto\Omega$ defined by
		$$
		\Phi(u):=G^{-1}(y-h(u))\cap\Omega\quad \text{for}\quad u\in \Omega.
		$$
		Clearly, $\Phi$ has acyclic values.  To see that $\gph\Phi$ is closed,
		pick any net $(u_\nu,x_\nu)$ in $\gph\Phi$ converging
		to a point $(u,x)\in\Omega\times\Omega$.  Then
		$$
		y-h(u_\nu)\in G(x_\nu).
		$$
		Since $h$ is continuous and $\gph(G|_{\Omega})$ is closed, we get
		$$
		y-h(u)\in G(x).
		$$
		Thus $\gph\Phi$ is closed, and Theorem~\ref{thm:EM} yields
		an $x\in\Omega$ such that $x\in\Phi(x)$, meaning that
		$y\in h(x)+G(x)$.
	\end{proof}
	
	Proposition~\ref{thm:ambient-range} is the form of the
	Eilenberg--Montgomery theorem used below.  It is recorded because the
	ambient-set formulation is convenient for the subsequent applications.
	Compare the convex-domain, convex-fibre formulation in
	\cite[Theorem~2.1]{CibulkaRoubalSetValued}.  Compact convex domains and
	nonempty convex or star-shaped correction fibres satisfy the stated
	topological assumptions, but convexity is not required here.
	
	\begin{corollary}
		\label{cor:error-set}
		Let $X$ and $Y$ be Fr\'echet spaces, let $\Omega\subset X$ be a nonempty
		compact metrizable acyclic ANR, and let $\Gamma\subset Y$ be nonempty.
		Let $G:X\tto Y$ be such that $G|_\Omega$ has a closed graph, let
		$f,g:\Omega\longrightarrow Y$ be continuous, and suppose that a set $\Xi\subset Y$
		satisfies
		\begin{enumerate}[label={\rm (\roman*)}]
			\item $(g-f)(\Omega)\subset\Xi$;
			\item $(g+G)(\Omega)\supset\Gamma+\Xi$;
			\item $(g+G)^{-1}(v)\cap\Omega$ is either empty or acyclic for every
			$v\in\Gamma+\Xi$.
		\end{enumerate}
		Then $(f+G)(\Omega)\supset\Gamma$.
	\end{corollary}
	
	\begin{proof}
		Define $H:\Omega\tto Y$ by $H(x):=g(x)+G(x)$; its graph is closed.
		Fix $y\in\Gamma$.  For every
		$u\in\Omega$, the point $y+(g-f)(u)$ belongs to $\Gamma+\Xi$; assumptions
		{\rm (ii)}--{\rm (iii)} make
		$$
		(g+G)^{-1}\bigl(y+(g-f)(u)\bigr)\cap\Omega
		$$
		nonempty and acyclic.  Apply Proposition~\ref{thm:ambient-range} with
		ambient space $\Omega$, mapping $H$, and perturbation $f-g$.
	\end{proof}
	
	The following scaled corollary is the form used in the subsequent
	arguments.
	
	\begin{corollary}
		\label{thm:scaled-transfer}
		Let $X$ be a Banach space and let $Y$ be a finite-dimensional normed
		space.  Let $S:Y\longrightarrow Y$ be a bounded linear
		isomorphism.  Consider a
		constant $\tau>0$, constants $\gamma,\delta\ge0$, a nonempty compact convex set
		$\Omega\subset X$, an operator $A\in\Lin(X,Y)$, a set-valued mapping
		$F:X\tto Y$ with a closed graph, and a continuous mapping
		$f:\Omega\longrightarrow Y$.  Assume that
		\begin{enumerate}
			
			\item[\rm (i)]
			$Au-f(u)\in S\bigl(\delta\tau\ball_Y\bigr)$ for every $u\in\Omega;$
			\item[\rm (ii)]
			$$
			(A+F)(\Omega)\supset
			S\bigl((\gamma+\delta)\tau\ball_Y\bigr);
			$$
			\item[\rm (iii)]
			$(A+F)^{-1}(v)\cap\Omega$
			is either empty or acyclic for every
			$v\in S((\gamma+\delta)\tau\ball_Y)$.
		\end{enumerate}
		Then
		$$
		(f+F)(\Omega)
		\supset S\bigl(\gamma\tau\ball_Y\bigr).
		$$
	\end{corollary}
	\begin{proof}
		Apply Corollary~\ref{cor:error-set}, with
		$$
		g(u):=Au,\quad G:=F,\quad
		\Gamma:=S(\gamma\tau\ball_Y),\quad\text{and}\quad
		\Xi:=S(\delta\tau\ball_Y).
		$$
		Conditions {\rm (i)}--{\rm (iii)} verify its assumptions, since
		$$
		\Gamma+\Xi=S\bigl((\gamma+\delta)\tau\ball_Y\bigr).
		$$
	\end{proof}
	
	Since acyclic sets are nonempty by definition, assumptions~{\rm (ii)}--
	{\rm (iii)} of Corollary~\ref{thm:scaled-transfer} are equivalent to the
	single requirement that every displayed fibre be acyclic.
	
	\begin{remark}[Hausdorff-contractive alternative]
		\label{rem:nadler-alternative}
		Let $(Z,d)$ be complete, let $\Omega\subset Z$ be nonempty and closed,
		let $Y$ be a metric vector space, let $G:Z\tto Y$, let
		$h:\Omega\longrightarrow Y$, and fix $y\in Y$.  Let
		$$
		\Phi_y(u):=G^{-1}\bigl(y-h(u)\bigr)\cap\Omega .
		$$
		If $\Phi_y(u)\in\CB(\Omega)$ for every $u\in\Omega$ and
		$d_H(\Phi_y(u),\Phi_y(v))\le q_y d(u,v)$ for some
		$q_y\in[0,1)$, the Nadler theorem \cite{Nadler} gives $x\in\Phi_y(x)$ and hence
		$y\in(h+G)(\Omega)$.  Thus this targetwise contraction can replace the
		compact-acyclic hypothesis.  In the scaled setting the correction is
		$$
		(A+F)^{-1}\bigl(y-(f(u)-Au)\bigr)\cap\Omega .
		$$
		For $y\in S(\gamma\tau\ball_Y)$, assumptions {\rm (i)}--{\rm (ii)}
		guarantee nonemptiness, but not the Hausdorff estimate.  The acyclic and
		contractive mechanisms are independent, and the Nadler theorem alone does
		not give uniqueness.
	\end{remark}
	\section{Perturbation stability}
	\label{sec:perturbation}
	
	This section records an ancillary fixed-base perturbation consequence of
	Corollary~\ref{cor:error-set}. It is not used in the proofs of the smooth or
	generalized-equation covering theorems below. Rather, it can be applied after
	an anisotropic inverse estimate has been established, provided that the
	relevant localized inverse fibres are acyclic.
	
	At first order, perturbation stability of linear openness and metric
	regularity is described by the Lyusternik--Graves theorem; see
	Graves~\cite{Graves} and, for set-valued extensions, Cibulka, Dontchev, and
	Veliov~\cite{CibulkaDontchevVeliov}. Related positive-order results can be
	found in \cite{YenYaoKienPositiveOrder,UderzoNonlinearRate,HeNgStability}.
	Throughout this section, $Y$ is finite-dimensional and is equipped with the
	structure of Definition~\ref{def:anisotropic-structure}.
	\begin{theorem}
		\label{thm:anisotropic-stability}
		Let $F:\R^n\tto Y$ be a set-valued mapping with a closed graph and
		closed domain, let $(\bx,\by)\in\gph F$, and let
		$\eta:\domm F\longrightarrow Y$. Assume that there is $r_c>0$ such that
		$\eta$ is continuous on
		$\domm F\cap\ball_{\R^n}[\bx,r_c]$ in the relative topology.
		
		Suppose that there exist $a,r_F>0$ such that
		\begin{equation}
			\by+Q_p(at)
			\subset
			F\bigl(\ball_{\R^n}[\bx,t]\bigr)
			\quad\text{for every}\quad t\in(0,r_F].
			\label{eq:base-anisotropic-openness}
		\end{equation}
		Assume that there exist $r_0,s>0$ such that
		$$
		F^{-1}(w)\cap\ball_{\R^n}[\bx,t]
		$$
		is either empty or acyclic for every $t\in(0,r_0]$ and every
		$w\in\by+Q_p(s)$.
		
		Finally, suppose that there exist
		$\ell\in[0,a)$ and  $r_\eta>0$
		such that
		\begin{equation}
			q_p\bigl(\eta(u)-\eta(\bx)\bigr)
			\le
			\ell\|u-\bx\|\quad\text{for every}\quad u\in\domm F\cap\ball_{\R^n}[\bx,r_\eta].
			\label{eq:local-qp-perturbation}
		\end{equation}
		Then there exists $r>0$ such that
		\begin{equation}
			\eta(\bx)+\by+Q_p\bigl((a-\ell)t\bigr)
			\subset
			(\eta+F)\bigl(\ball_{\R^n}[\bx,t]\bigr)
			\quad\text{for every}\quad t\in(0,r].
			\label{eq:perturbed-anisotropic-openness}
		\end{equation}
		Consequently, there exists a neighborhood $W$ of
		$\eta(\bx)+\by$ such that
		$$
		\dist\bigl(\bx,(\eta+F)^{-1}(y)\bigr)
		\le
		\frac{1}{a-\ell}\,
		q_p\bigl(y-\eta(\bx)-\by\bigr)
		\quad\text{for every}\quad y\in W.
		$$
	\end{theorem}
	
	\begin{proof}
		Set $c:=a-\ell>0$ and choose
		$r\in(0,\min\{r_F,r_0,r_\eta,r_c\}]$ such that $a r\le s$.
		
		Fix $t\in(0,r]$ and set
		$$
		\Omega_t:=\ball_{\R^n}[\bx,t]\quad\text{and}\quad
		K_t:=\domm F\cap\Omega_t.
		$$
		The set $\Omega_t$ is compact and convex, while $K_t$ is closed in
		$\Omega_t$. Moreover, \eqref{eq:local-qp-perturbation} gives
		$$
		\eta(K_t)
		\subset
		\eta(\bx)+Q_p(\ell t).
		$$
		Since $Q_p(\ell t)$ is convex, the Dugundji extension theorem
		\cite{Dugundji} provides a continuous extension
		$\widetilde\eta_t:
		\Omega_t\longrightarrow\eta(\bx)+Q_p(\ell t)$
		of $\eta\vert_{K_t}$.
		
		Define
		$$
		G:=F-\by,
		\quad
		k_t:=\widetilde\eta_t-\eta(\bx),
		\quad
		\Gamma_t:=Q_p(ct),
		\quad\text{and}\quad
		\Xi_t:=Q_p(\ell t).
		$$
		Then
		$$
		-k_t(\Omega_t)\subset\Xi_t.
		$$
		Since $c+\ell=a$ and $q_p$ is subadditive,
		$$
		\Gamma_t+\Xi_t
		\subset Q_p(at).
		$$
		Furthermore, \eqref{eq:base-anisotropic-openness} yields
		$$
		Q_p(at)\subset G(\Omega_t),
		$$
		and the choice of $r$ gives
		$$
		Q_p(at)\subset Q_p(ar)\subset Q_p(s).
		$$
		
		For every $v\in\Gamma_t+\Xi_t$, the localized fibre
		$$
		G^{-1}(v)\cap\Omega_t
		=
		F^{-1}(\by+v)\cap\Omega_t
		$$
		is therefore either empty or acyclic. Corollary~\ref{cor:error-set}, with
		$f:=k_t$, $g:=0,$ $\Gamma:=\Gamma_t,$ and $\Xi:=\Xi_t,$
		implies
		$$
		(k_t+G)(\Omega_t)\supset Q_p(ct).
		$$
		Whenever $u\in\Omega_t$ and $G(u)\ne\varnothing$, we have $u\in K_t$ and hence
		$\widetilde\eta_t(u)=\eta(u)$. Consequently,
		$$
		(\eta+F)(\Omega_t)
		\supset
		\eta(\bx)+\by+Q_p(ct),
		$$
		which proves \eqref{eq:perturbed-anisotropic-openness}.
		
		The inverse estimate follows from the equivalence of anisotropic
		openness and semiregularity in
		Remark~\ref{rem:anisotropic-covering-inverse-equivalence}, with
		openness rate $c=a-\ell$.
	\end{proof}
	
	\begin{remark}
		For $p=1$, the gauge $q_1$ is a norm, the perturbed openness rate is
		$a-\ell$, and the corresponding semiregularity constant is
		$$
		\frac{1}{a-\ell}.
		$$
		If the original openness rate is written as $a=1/M$, this becomes
		$$
		\frac{M}{1-M\ell},
		$$
		which is the familiar fixed-base Lyusternik--Graves perturbation bound.
		Theorem~\ref{thm:anisotropic-stability} remains a fixed-base openness
		and semiregularity result.
	\end{remark}
	
	\begin{remark}
		Condition~\eqref{eq:local-qp-perturbation} is equivalent to
		$$
		\|P_i(\eta(x)-\eta(\bx))\|
		\le
		\ell^i\|x-\bx\|^i,
		\qquad i=1,\ldots,p,
		$$
		on the same relative neighborhood.  Thus the $i$th component of the
		perturbation has the order required by the anisotropic scaling.  Standard
		Lipschitz continuity is generally insufficient for the components with
		$i\ge2$.
		
		The topological assumption on the localized inverse fibres cannot in
		general be deduced from fixed-base anisotropic openness, or equivalently
		from anisotropic semiregularity, alone.  It is
		automatic, for example, when all the relevant nonempty localized fibres
		are singletons.
	\end{remark}

	We now derive fixed-base anisotropic openness from the smooth triangular
	higher-order model.
	
	\section{Anisotropic covering relative to a triangular decomposition}
	\label{sec:p-factor}
	
	\subsection{Triangular $p$-factor operators and scaled approximation estimates}
	
	Throughout this section, let $p\ge2$, let $X$ be a Banach space, and let
	$Y$ be a Banach space equipped with the anisotropic range structure of
	Definition~\ref{def:anisotropic-structure}.  Let $U\subset X$ be open, let
	$f:U\longrightarrow Y$ be of class $C^p$, and fix $\bx\in U$.  We impose
	the triangular factor condition relative to the chosen decomposition
	\begin{equation}
		P_i f^{(j)}(\bx)=0
		\quad\text{whenever}\quad 1\le j<i\le p.
		\label{eq:P1}
	\end{equation}
	
	Throughout Sections~\ref{sec:p-factor}--\ref{sec:process-criterion}, all
	factor objects and terminology are relative to this fixed triangular
	decomposition and its projections $P_i$.  Condition~\eqref{eq:P1} alone
	does not assert that the decomposition is a recursive $p$-factor
	decomposition used in standard $p$-regularity theory; see
	\cite[Section~3, equations~(15)--(19)]{BednarczukEtAlSurvey}.
	The results are conditional on this one indexed choice and neither
	construct nor optimize it; after an arbitrary change of decomposition,
	the hypotheses must be verified again.  The reduced-index and zero-block
	conventions of Remark~\ref{rem:grading-dependence} remain in force.
	
	For $h\in X$ and $i\in\{1,\ldots,p\}$, define
	$L_i(h):X\longrightarrow Y_i$ by
	\begin{equation*}
		L_i(h)\xi
		:=\frac{1}{(i-1)!}
		P_i f^{(i)}(\bx)[h]^{i-1}\xi,
	\end{equation*}
	where $[h]^{i-1}\xi$ means that the first $i-1$ arguments of
	$f^{(i)}(\bx)$ are equal to $h$ and the last argument is equal to $\xi$,
	with $L_1(h)=P_1f'(\bx)$.  The \emph{normalized triangular $p$-factor
		operator relative to the chosen decomposition} is
	\begin{equation*}
		\Psi_p(h):=\sum_{i=1}^pL_i(h);
	\end{equation*}
	note that $\Psi_p(h):X\longrightarrow Y$.
	We refer to the surjectivity of $\Psi_p(h)$ as the \emph{triangular factor
		surjectivity condition relative to the chosen decomposition}.  The direction
	$h$ satisfies the \emph{triangular $p$-kernel condition relative to the
		chosen decomposition} if
	\begin{equation}
		\Psi_p(h)h=0.
		\label{eq:P4}
	\end{equation}
	This condition is equivalent to
	$$
	P_i f^{(i)}(\bx)[h]^i=0
	\quad\text{for every}\quad i\in\{1,\ldots,p\}.
	$$
	\begin{remark}[The graded Taylor model]
		\label{rem:sussmann-dictionary}
		Define
		$$
		\mathcal T_p(z)
		:=
		\sum_{i=1}^p\frac{1}{i!}P_i f^{(i)}(\bx)[z]^i,
		\qquad z\in X.
		$$
		Then, for every $s>0$,
		$$
		\mathcal T_p(sz)
		=
		\sum_{i=1}^p s^iP_i\mathcal T_p(z)
		=
		sD_s\mathcal T_p(z).
		$$
		Moreover,
		$$
		D\mathcal T_p(h)=\Psi_p(h).
		$$
		Under the triangular kernel condition \eqref{eq:P4}, we also have
		$$
		\mathcal T_p(h)=0.
		$$
		Thus, whenever $\Psi_p(h)$ is surjective, $h$ is a regular zero of
		the graded Taylor model, with ordinary input scaling and anisotropic
		target scaling $y\mapsto sD_s y$.
	\end{remark}
	If the chosen blocks and projections do arise from a standard recursive
	construction and $\widehat\Psi_p(h)$ denotes the standard unnormalized
	$p$-factor operator \cite[Definition~4]{BednarczukEtAlSurvey}, then
	$$
	\widehat\Psi_p(h)=J\circ\Psi_p(h),
	\qquad
	J:=\sum_{i=1}^p(i-1)!P_i.
	$$
	Here $J$ is an invertible diagonal operator.
	Consequently, in that special case the preceding surjectivity and kernel
	conditions agree with the corresponding standard ones; no such
	identification is used in the results below.
	
	For $t>0$, define the scaled triangular factor operator
	\begin{equation}
		A_t:=D_t\circ\Psi_p(h)=\sum_{i=1}^pt^{i-1}L_i(h).
		\label{eq:P5}
	\end{equation}
	
	Let $x_t:=\bx+th.$
	For the next lemma, let $E\subset X$ be a finite-dimensional
	linear subspace.
	
	The smooth data and the notation $L_i(h)$, $\Psi_p(h)$, $D_t$, $A_t$, and
	$x_t$ introduced above remain in force throughout
	Sections~\ref{sec:partial-p-factor} and~\ref{sec:process-criterion}.
	
	\begin{lemma}
		\label{lem:p-factor-scaled-error}
		Assume \eqref{eq:P1}, and let $E\subset X$ be a finite-dimensional
		linear subspace. For every $\varepsilon>0$ there exist $\rho,r>0$ such
		that
		$$
		x_t+\rho t\ball_E\subset U
		\quad\text{for every}\quad t\in(0,r],
		$$
		and the following assertions hold:
		\begin{enumerate}[label={\rm (\roman*)}]
			
			\item For every $t\in(0,r]$ and every $\xi\in\rho t\ball_E$, we have
			\begin{equation}
				\left\|D_t^{-1}
				\bigl(f(x_t+\xi)-f(x_t)-A_t\xi\bigr)
				\right\|_\Sigma
				\le\varepsilon\|\xi\|.
				\label{eq:P7}
			\end{equation}
			
			\item If, in addition, the triangular kernel condition
			\eqref{eq:P4} holds, then, after decreasing $r$ if necessary,
			\begin{equation}
				\left\|D_t^{-1}
				\bigl(f(x_t)-f(\bx)\bigr)
				\right\|_\Sigma
				\le\varepsilon\rho t
				\quad\text{for every}\quad t\in(0,r].
				\label{eq:P10}
			\end{equation}
			
		\end{enumerate}
		Consequently, under \eqref{eq:P4},
		\begin{equation}
			\left\|D_t^{-1}
			\bigl(f(x_t+\xi)-f(\bx)-A_t\xi\bigr)
			\right\|_\Sigma
			\le 2\varepsilon\rho t
			\label{eq:centered-scaled-error}
		\end{equation}
		for every $t\in(0,r]$ and every $\xi\in\rho t\ball_E$.
	\end{lemma}
	
	\begin{proof}
		Choose $\rho_0,\delta>0$ such that
		$$
		\bx+\delta\ball_X\subset U,
		$$
		and let
		$$
		K_0:=h+\rho_0\ball_E.
		$$
		Since $E$ is finite-dimensional, $K_0$ is compact. After decreasing
		$r_0>0$, we have
		$$
		\bx+tw\in U
		\quad\text{for}\quad
		w\in K_0 \quad\text{and}\quad t\in(0, r_0].
		$$
		
		Taylor's formula, together with \eqref{eq:P1}, gives, uniformly for
		$w\in K_0$,
		$$
		t^{1-i}P_i f'(\bx+tw)|_E
		=
		L_i(w)|_E+o(1)
		\quad\text{as}\quad t\downarrow0,
		$$
		for every $i=1,\ldots,p$. Indeed, for $i\ge2$,
		$$
		\begin{aligned}
			&t^{1-i}P_i f'(\bx+tw)|_E-L_i(w)|_E\\
			&\quad=
			\frac{1}{(i-2)!}
			\int_0^1(1-s)^{i-2}
			P_i\bigl(f^{(i)}(\bx+stw)-f^{(i)}(\bx)\bigr)
			[w]^{i-1}|_E\,ds,
		\end{aligned}
		$$
		while the case $i=1$ follows directly from the continuity of $f'$.
		Consequently,
		$$
		\sup_{w\in K_0}
		\left\|
		D_t^{-1}f'(\bx+tw)|_E-\Psi_p(w)|_E
		\right\|_\Sigma
		\longrightarrow0
		\quad\text{as}\quad  t\downarrow0.
		$$
		
		By continuity of $w\mapsto\Psi_p(w)|_E$, choose
		$\rho\in(0,\rho_0]$ and then $r\in(0,r_0]$ so that
		$$
		\sup_{z\in\rho\ball_E}
		\left\|
		D_t^{-1}f'\bigl(\bx+t(h+z)\bigr)|_E
		-\Psi_p(h)|_E
		\right\|_\Sigma
		\le\varepsilon
		$$
		for every $t\in(0,r]$.
		
		Let $\xi\in\rho t\ball_E$. Since
		$D_t^{-1}A_t=\Psi_p(h)$, the fundamental theorem of calculus gives
		$$
		\begin{aligned}
			D_t^{-1}
			\bigl(f(x_t+\xi)-f(x_t)-A_t\xi\bigr)
			=
			\int_0^1
			\left[
			D_t^{-1}f'(x_t+s\xi)|_E-\Psi_p(h)|_E
			\right]\xi\,ds.
		\end{aligned}
		$$
		Taking norms proves \eqref{eq:P7}.
		
		Suppose now that \eqref{eq:P4} holds. For every
		$i=1,\ldots,p$, Taylor's formula and \eqref{eq:P1} yield
		$$
		P_i\bigl(f(\bx+th)-f(\bx)\bigr)
		=
		\frac{t^i}{(i-1)!}
		\int_0^1(1-s)^{i-1}
		P_i f^{(i)}(\bx+sth)[h]^i\,ds.
		$$
		Condition \eqref{eq:P4} implies
		$$
		P_i f^{(i)}(\bx)[h]^i=0,
		$$
		and hence
		$$
		P_i\bigl(f(\bx+th)-f(\bx)\bigr)=o(t^i).
		$$
		Therefore
		\begin{equation}
			\label{eq:p-factor-center-small-o}
			\left\|D_t^{-1}
			\bigl(f(x_t)-f(\bx)\bigr)
			\right\|_\Sigma
			=
			\sum_{i=1}^p
			t^{1-i}
			\left\|P_i\bigl(f(x_t)-f(\bx)\bigr)\right\|
			=o(t).
		\end{equation}
		Since $\rho>0$ is already fixed, decreasing $r$ proves
		\eqref{eq:P10}.
		
		Finally, \eqref{eq:P7}, \eqref{eq:P10}, and
		$\|\xi\|\le\rho t$ give
		$$
		\begin{aligned}
			\left\|D_t^{-1}
			\bigl(f(x_t+\xi)-f(\bx)-A_t\xi\bigr)
			\right\|_\Sigma\le
			\varepsilon\|\xi\|+\varepsilon\rho t
			\le2\varepsilon\rho t,
		\end{aligned}
		$$
		which proves \eqref{eq:centered-scaled-error}.
	\end{proof}
	
	The preceding lemma controls both errors required for range transfer:
	the nonlinear increment on the moving tube and the displacement of its
	centre. In particular, \eqref{eq:centered-scaled-error} provides a single
	scaled approximation estimate centred at $f(\bx)$. We now combine this
	estimate with triangular factor surjectivity.
	
	\subsection{Covering and componentwise inverse estimates}
	
	\begin{theorem}
		\label{thm:p-factor-covering}
		Let $Y$ be finite-dimensional. Assume \eqref{eq:P1} and \eqref{eq:P4},
		and suppose that $\Psi_p(h)$ is surjective.
		Then, for every finite-dimensional subspace $E\subset X$ satisfying
		$$
		\Psi_p(h)(E)=Y,
		$$
		there exist constants $c,\rho,r>0$ such that, for every
		$t\in(0,r]$,
		$$
		\bx+th+\rho t\ball_E\subset U
		$$
		and
		$$
		f(\bx)+D_t(ct\ball_Y^\Sigma)
		\subset
		f(\bx+th+\rho t\ball_E).
		$$
	\end{theorem}
	
	\begin{proof}
		Fix a finite-dimensional subspace $E\subset X$ such that
		$\Psi_p(h)(E)=Y$, and let $R:Y\longrightarrow E$ be a right inverse of
		$\Psi_p(h)\vert_E$. Set
		$$
		C:=\max\left\{
		1,\|R\|_{(Y,\|\cdot\|_\Sigma)\to X}
		\right\}.
		$$
		Apply Lemma~\ref{lem:p-factor-scaled-error} with
		$ \varepsilon:=\tfrac{1}{4C}.$
		It provides $\rho,r>0$ such that
		$$
		x_t+\rho t\ball_E\subset U
		\quad\text{for every}\quad t\in(0,r],
		$$
		and, by \eqref{eq:centered-scaled-error},
		\begin{equation}
			\left\|D_t^{-1}
			\bigl(f(x_t+\xi)-f(\bx)-A_t\xi\bigr)
			\right\|_\Sigma
			\le \frac{\rho}{2C}t
			\label{eq:centered-error-for-covering}
		\end{equation}
		for every $t\in(0,r]$ and every $\xi\in\rho t\ball_E$.
		
		Fix $t\in(0,r]$ and set
		$$
		\Omega_t:=\rho t\ball_E
		\quad\text{and}\quad
		\widetilde\phi_t(\xi):=f(x_t+\xi)-f(\bx),
		\quad \xi\in\Omega_t.
		$$
		Since $A_t=D_t\Psi_p(h)$ and $\|R\|\le C$,
		$$
		A_t(\Omega_t)
		\supset
		D_t\left(\frac{\rho}{C}t\ball_Y^\Sigma\right).
		$$
		Indeed, if $\|z\|_\Sigma\le\rho t/C$, then
		$\xi:=Rz$ belongs to $\Omega_t$ and $A_t\xi=D_tz$.
		Every fibre of $A_t\vert_{\Omega_t}$ over the set on the right is
		nonempty, compact, and convex.
		
		Moreover, \eqref{eq:centered-error-for-covering} gives
		$$
		A_t\xi-\widetilde\phi_t(\xi)
		\in
		D_t\left(\frac{\rho}{2C}t\ball_Y^\Sigma\right)
		\quad\text{for every}\quad \xi\in\Omega_t.
		$$
		Corollary~\ref{thm:scaled-transfer}, applied with domain space $E$, with
		$Y$ endowed with $\|\cdot\|_\Sigma$, and with
		$$
		S:=D_t,\quad
		\tau:=t,\quad
		A:=A_t\vert_E,\quad
		F\equiv\{0\},\quad
		f:=\widetilde\phi_t,\quad
		\gamma=\delta:=\frac{\rho}{2C},
		$$
		therefore yields
		$$
		\widetilde\phi_t(\Omega_t)
		\supset
		D_t\left(\frac{\rho}{2C}t\ball_Y^\Sigma\right).
		$$
		Thus
		$$
		f(\bx)+D_t\left(\frac{\rho}{2C}t\ball_Y^\Sigma\right)
		\subset
		f(x_t+\Omega_t)
		=
		f(\bx+th+\rho t\ball_E).
		$$
		The conclusion follows with $c:=\frac{\rho}{2C}.$
	\end{proof}
	The moving-tube inclusion also yields fixed-base anisotropic openness.
	Indeed, set
	$$
	L:=\|h\|+\rho
	$$
	and choose $\kappa>0$ such that
	$$
	\sum_{i\in\mathcal I}\kappa^i\le c.
	$$
	After decreasing $r$ if necessary, assume that
	$$
	\ball_X[\bx,Lr]\subset U.
	$$
	If $t\in(0,r]$ and $d\in Q_p(\kappa t)$, then
	$$
	\|D_t^{-1}d\|_\Sigma
	=
	\sum_{i\in\mathcal I}t^{1-i}\|P_i d\|
	\le
	t\sum_{i\in\mathcal I}\kappa^i
	\le ct.
	$$
	Consequently,
	$$
	Q_p(\kappa t)
	\subset
	D_t(ct\ball_Y^\Sigma).
	$$
	Moreover,
	$$
	\bx+th+\rho t\ball_E
	\subset
	\ball_X[\bx,Lt].
	$$
	Letting $s:=Lt$ and $a:=\kappa/L$, we obtain
	\begin{equation}
		f(\bx)+Q_p(as)
		\subset
		f\bigl(\ball_X[\bx,s]\bigr)
		\quad\text{for every}\quad s\in(0,Lr].
		\label{eq:p-factor-fixed-base-openness}
	\end{equation}
	Thus $f$ is anisotropically open at $(\bx,f(\bx))$ with rate $a$.
	By Remark~\ref{rem:anisotropic-covering-inverse-equivalence}, it is also
	anisotropically semiregular there. More precisely, on the neighborhood
	$$
	W:=f(\bx)+\{z\in Y:q_p(z)<\kappa r\},
	$$
	we have
	$$
	\dist\bigl(\bx,f^{-1}(y)\bigr)
	\le
	\frac{L}{\kappa}\,
	q_p\bigl(y-f(\bx)\bigr)
	\quad\text{for every}\quad y\in W.
	$$
	Since $p=\max\mathcal I$, after possibly shrinking $W$ this also yields
	$$
	\dist\bigl(\bx,f^{-1}(y)\bigr)
	\le
	\widetilde M\|y-f(\bx)\|^{1/p}
	\quad\text{for every}\quad y\in W
	$$
	with some $\widetilde M>0$.

	\begin{remark}[A logarithmic $C^p$ remainder without a strict power gap]
		\label{rem:logarithmic-remainder}
		Let $p\ge2$, put
		$$
		U:=\{x\in\mathbb R^3:\|x\|<1/2\},
		\qquad
		\mathfrak m(x):=x_1^{p-1}x_2,
		$$
		and define
		$$
		g(x):=\frac{\mathfrak m(x)}{\log(e/\|x\|)}
		\quad (x\ne0),
		\qquad
		g(0):=0.
		$$
		The logarithmic factor is a standard slowly varying correction at the
		origin; for background on slowly varying functions, see
		\cite[Section~1.5]{BinghamGoldieTeugels}. Put
		$L(x):=1/\log(e/\|x\|)$ for $x\ne0$. For every multi-index $\gamma$
		with $|\gamma|\ge1$, radial differentiation gives
		$$
		|D^\gamma L(x)|
		\le
		\frac{C_\gamma}{\|x\|^{|\gamma|}\log(e/\|x\|)}.
		$$
		Since $\mathfrak m$ is homogeneous of degree $p$, Leibniz's rule yields
		$$
		|D^\beta g(x)|
		\le
		C_\beta\frac{\|x\|^{p-|\beta|}}{\log(e/\|x\|)}
		\qquad
		(0<\|x\|<1/2,\ |\beta|\le p).
		$$
		The right-hand side tends to zero for $|\beta|\le p$; the standard
		removable-singularity argument applied successively to the derivatives
		therefore gives $g\in C^p(U)$ and $D^\beta g(0)=0$ for every
		$|\beta|\le p$.
		For
		$$
		f(x):=\bigl(x_3,\mathfrak m(x)+g(x)\bigr),
		$$
		with $Y_1=\operatorname{span}\{(1,0)\}$,
		$Y_p=\operatorname{span}\{(0,1)\}$, and the intermediate blocks equal
		to zero, the choice $h=e_1$ and
		$E=\operatorname{span}\{e_2,e_3\}$ gives
		$$
		\Psi_p(h)\xi=(\xi_3,\xi_2),
		\qquad
		\Psi_p(h)h=0,
		\qquad
		\Psi_p(h)(E)=\mathbb R^2.
		$$
		Thus Theorem~\ref{thm:p-factor-covering} applies without any change in
		the triangular factor operator.
		
		Nevertheless, no strict weighted power gap is available. Indeed, for a
		positive weight vector $\lambda=(\lambda_1,\lambda_2,\lambda_3)$, set
		$$
		\delta_t^\lambda x
		:=
		(t^{\lambda_1}x_1,t^{\lambda_2}x_2,t^{\lambda_3}x_3),
		\qquad
		\alpha:=(p-1)\lambda_1+\lambda_2.
		$$
		With $u=(1,1,1)$ and
		$\lambda_*:=\min\{\lambda_1,\lambda_2,\lambda_3\}$, for every
		$\gamma>0$ we have
		$$
		t^{-(\alpha+\gamma)}
		|g(\delta_t^\lambda u)|
		=
		\frac{t^{-\gamma}}
		{\lambda_*\log(1/t)+O(1)}
		\longrightarrow+\infty.
		$$
		Consequently, the remainder is not
		$O(t^{\alpha+\gamma})$ along this ray for any $\gamma>0$ and therefore
		falls outside the strict weighted power-gap subclass covered by
		\cite[Theorem~3.2]{ArutyunovZhukovskiySbornik2025}; see also
		\cite{ArutyunovZhukovskiyMathNotes2025}. Nevertheless, the present $C^p$
		selected-direction theorem applies.
	\end{remark}
	
	\section{Transfer through an exact partial model}
	\label{sec:partial-p-factor}
	
	\subsection{Finite-scale transfer from the exact partial model}
	
	We now pass from $f$ to $G=f+F$ while retaining the exact value
	$F(x_t+\xi)$ in the model. The finite-scale result below transfers an
	assumed model covering, together with acyclicity of its localized fibres,
	to the original generalized equation. It is an auxiliary template; a
	sufficient convex-process criterion is developed in
	Section~\ref{sec:process-criterion}.
	
	Let the standing assumptions of the preceding section hold, assume that
	$Y$ is finite-dimensional, let $F:X\tto Y$ have a closed graph, choose
	$\bv\in F(\bx)$, and let
	$$
	G:U\tto Y,\quad G(x):=f(x)+F(x),
	\quad
	\by:=f(\bx)+\bv,
	\quad\text{and}\quad
	x_t:=\bx+th.
	$$
	Given a finite-dimensional subspace $E\subset X$, $\rho>0$, and
	$v_t\in F(x_t)$ such that $x_t\in U$, define
	$$
	\Omega_{t,\rho}:=\rho t\ball_E,
	\quad
	y_t:=f(x_t)+v_t,
	$$
	and
	\begin{equation}
		\mathcal G_t^p(x_t+\xi)
		:=f(x_t)+A_t\xi+F(x_t+\xi)\quad\text{for}
		\quad \xi\in E,\quad\text{and}\quad x_t+\xi\in U.
		\label{eq:PG1}
	\end{equation}
	Only the single-valued increment is replaced by its triangular factor model
	relative to the chosen decomposition.
	
	\begin{theorem}
		\label{thm:p-partial-finite}
		Let $E\subset X$ be finite-dimensional, let $t,\rho,c,\ell>0$, and let
		$v_t\in F(x_t)$.  Suppose that $x_t+\Omega_{t,\rho}\subset U$ and that
		\begin{enumerate}[label={\rm (\roman*)}]
			\item for every $\xi\in\Omega_{t,\rho}$ we have
			
			$$
			\left\|D_t^{-1}\bigl(
			f(x_t+\xi)-f(x_t)-A_t\xi
			\bigr)\right\|_{\Sigma}
			\le\ell\rho t;
			$$
			\item
			$$
			\mathcal G_t^p(x_t+\Omega_{t,\rho})
			\supset
			y_t+D_t((c+\ell)\rho t\ball_Y^{\Sigma});
			$$
			\item for every $w\in D_t((c+\ell)\rho t\ball_Y^{\Sigma})$, the set
			$$
			\{\xi\in\Omega_{t,\rho}:
			y_t+w\in\mathcal G_t^p(x_t+\xi)\}
			$$
			is either empty or acyclic.
		\end{enumerate}
		Then
		\begin{equation}
			y_t+D_t(c\rho t\ball_Y^{\Sigma})
			\subset(f+F)(x_t+\Omega_{t,\rho}).
			\label{eq:PG5}
		\end{equation}
	\end{theorem}
	
	\begin{proof}
		Let $\phi_t:\Omega_{t,\rho}\longrightarrow Y$ and $F_t:E\tto Y$ be given by
		$$
		\phi_t(\xi):=f(x_t+\xi)-f(x_t)
		\quad\text{and}\quad
		F_t(\xi):=F(x_t+\xi)-v_t.
		$$
		The mapping $F_t$ has a closed graph.
		By {\rm (i)},
		$$
		A_t\xi-\phi_t(\xi)
		\in D_t(\ell\rho t\ball_Y^\Sigma)
		\quad\text{for each}\quad\xi\in\Omega_{t,\rho}.
		$$
		After translation by $y_t=f(x_t)+v_t$, {\rm (ii)}--{\rm (iii)} give
		$$
		(A_t+F_t)(\Omega_{t,\rho})
		\supset D_t((c+\ell)\rho t\ball_Y^\Sigma)
		$$
		and the corresponding empty-or-acyclic fibre condition. Corollary~%
		\ref{thm:scaled-transfer}, applied with domain space $E$, with $Y$ endowed
		with $\|\cdot\|_\Sigma$, and with
		$$
		\Omega:=\Omega_{t,\rho},\quad S:=D_t,\quad \tau:=\rho t,\quad
		A:=A_t\vert_E,\quad F:=F_t,\quad f:=\phi_t,\quad
		\gamma:=c,\quad\delta:=\ell,
		$$
		yields
		$$
		(\phi_t+F_t)(\Omega_{t,\rho})
		\supset D_t(c\rho t\ball_Y^\Sigma).
		$$
		Adding $y_t$ gives \eqref{eq:PG5}.
	\end{proof}
	
	Theorem~\ref{thm:p-partial-finite} is a single-scale statement centred
	at $y_t$. Lemma~\ref{lem:p-factor-scaled-error} makes its smooth
	approximation error arbitrarily small on a sufficiently thin tube. What
	remains is a model covering with a scale-independent positive margin,
	acyclic model fibres, and an anisotropically negligible displacement of
	the centre. The next definition packages precisely these requirements.
	
	\subsection{An asymptotic exact-model transfer condition}
	
	\begin{definition}
		\label{def:exact-partial-transfer}
		The generalized equation $G=f+F$ is said to \emph{satisfy the exact
			partial-model transfer condition relative to the chosen decomposition}
		at $(\bx,\by)$ with respect to $h$ if there exist a
		finite-dimensional subspace $E\subset X$, constants $a,\rho_0,r_0>0$,
		a choice $v_t\in F(x_t)$ for every $t\in(0,r_0]$, and, for every
		$\rho\in(0,\rho_0]$, a radius $r_\rho\in(0,r_0]$, such that
		$x_t+\rho_0t\ball_E\subset U$ for every $t\in(0,r_0]$ and, with
		$y_t=f(x_t)+v_t$,
		\begin{enumerate}[label={\rm (\roman*)}]
			\item
			\begin{equation*}
				\frac{\|D_t^{-1}(y_t-\by)\|_{\Sigma}}{t}\to0
				\quad\text{as}\quad t\downarrow0;
			\end{equation*}
			\item for every $\rho\in(0,\rho_0]$ and every $t\in(0,r_\rho]$, we have
			$$
			\mathcal G_t^p(x_t+\rho t\ball_E)
			\supset y_t+D_t(a\rho t\ball_Y^\Sigma);
			$$
			\item for every $\rho\in(0,\rho_0]$, every $t\in(0,r_\rho]$, and every
			$w\in D_t(a\rho t\ball_Y^{\Sigma})$, the set
			\begin{equation*}
				\{\xi\in\rho t\ball_E:
				y_t+w\in\mathcal G_t^p(x_t+\xi)\}
			\end{equation*}
			is either empty or acyclic.
		\end{enumerate}
	\end{definition}
	The terminology emphasizes that conditions {\rm (ii)}--{\rm (iii)} already
	concern the covering and the topology of the exact model. They are inputs to the
	transfer argument, not an independent regularity characterization of $F$.
	
	\begin{remark}[Elementary checks for Definition~\ref{def:exact-partial-transfer}]
		\label{rem:partial-elementary-checks}
		The conditions in Definition~\ref{def:exact-partial-transfer} can be verified separately.
		
		First, the tube condition is automatic after decreasing $r_0$. Indeed,
		if $\bx+d\ball_X\subset U$, then
		$$
		x_t+\rho_0t\ball_E
		\subset
		\bx+t(\|h\|+\rho_0)\ball_X
		\subset U
		$$
		whenever $r_0(\|h\|+\rho_0)\le d$.
		
		For the centring condition, observe that
		$$
		\frac{\|D_t^{-1}(y_t-\by)\|_\Sigma}{t}
		=
		\sum_{i=1}^p
		\frac{
			\|P_i(f(x_t)-f(\bx)+v_t-\bv)\|
		}{t^i}.
		$$
		Consequently, condition~{\rm (i)} holds if and only if
		$$
		P_i\bigl(f(x_t)-f(\bx)+v_t-\bv\bigr)=o(t^i),
		\qquad i=1,\ldots,p.
		$$
		If the triangular kernel condition \eqref{eq:P4} holds, then
		$$
		P_i(f(x_t)-f(\bx))=o(t^i),
		$$
		and condition~{\rm (i)} is therefore equivalent to
		$$
		P_i(v_t-\bv)=o(t^i),
		\qquad i=1,\ldots,p.
		$$
		In particular, it is automatic if $\bv\in F(x_t)$ for all sufficiently
		small $t$ and one chooses $v_t=\bv$.
		
		Finally, condition~{\rm (iii)} is automatic if $\gph F$ is convex.
		Indeed, the fibre occurring there can be written as
		$$
		\rho t\ball_E\cap
		\bigl\{
		\xi\in E:
		(x_t+\xi,v_t+w-A_t\xi)\in\gph F
		\bigr\}.
		$$
		It is a closed convex subset of the compact finite-dimensional ball
		$\rho t\ball_E$. Hence it is either empty or a nonempty compact convex
		set; in the latter case it is acyclic. Notice that convexity of the
		individual values
		$F(x)$ alone is not sufficient; convexity of the relevant part of
		$\gph F$, or directly of the localized fibres, is needed.
	\end{remark}
	Since an acyclic set is nonempty, conditions~{\rm (ii)}--{\rm (iii)}
	are equivalent to requiring every fibre displayed in~{\rm (iii)} to be
	acyclic.  We use the split form to distinguish model covering from the
	topological condition.
	
	\begin{theorem}
		\label{thm:exact-partial-transfer}
		Suppose that $G:=f+F$ satisfies the exact partial-model transfer condition
		relative to the chosen decomposition at $(\bx,\by)$ with respect to $h$.
		Then there exist a finite-dimensional
		subspace $E\subset X$ and constants $c,\rho,r>0$ such that
		\begin{equation}
			\by+D_t(ct\ball_Y^{\Sigma})
			\subset(f+F)(\bx+th+\rho t\ball_E)
			\quad\text{for every}\quad t\in(0,r].
			\label{eq:PG9}
		\end{equation}
	\end{theorem}
	
	\begin{proof}
		Let $E,a,\rho_0,r_0$ and the radii $r_\rho$ be as in
		Definition~\ref{def:exact-partial-transfer}. Apply
		Lemma~\ref{lem:p-factor-scaled-error} with $\varepsilon=a/4$ and choose
		$\rho\in(0,\rho_0]$ and $r_1>0$ so that its estimate holds for every
		$t\in(0,r_1]$.  Choose $r\in(0,\min\{r_1,r_\rho\}]$ sufficiently small
		that
		$$
		\|D_t^{-1}(y_t-\by)\|_{\Sigma}
		\le\frac{a\rho}{4}t
		\quad\text{for every} \quad t\in(0,r].
		$$
		For every such $t$, Theorem~\ref{thm:p-partial-finite}, with
		$c:=3a/4$ and $\ell:=a/4$, yields
		$$
		y_t+D_t\left(\frac{3a\rho}{4}t\ball_Y^{\Sigma}\right)
		\subset(f+F)(x_t+\rho t\ball_E).
		$$
		Therefore
		$$
		\by+D_t\left(\frac{a\rho}{2}t\ball_Y^\Sigma\right)
		\subset
		y_t+D_t\left(\frac{3a\rho}{4}t\ball_Y^\Sigma\right)
		\subset(f+F)(x_t+\rho t\ball_E),
		$$
		which proves \eqref{eq:PG9} with $c:=a\rho/2$.
	\end{proof}
	
	\begin{remark}[Model fibres]
		No graphical derivative or proto-derivative of $F$ is used: the
		set-valued mapping enters the model through the exact value $F(x_t+\xi)$.  If
		$\gph F$ is convex, the model fibres are convex; hence the
		empty-or-acyclic fibre condition is automatic, while model covering remains
		a separate assumption.
	\end{remark}
	
	We next replace the assumed exact-model covering by a sufficient condition. A closed convex process is embedded into $F$ through an inner
	approximation. Surjectivity of its sum with $\Psi_p(h)$ supplies the
	covering, while convexity supplies the acyclicity required by the transfer
	argument.

	\section{A convex-process criterion for generalized equations}
	\label{sec:process-criterion}
	
	\subsection{Covering by surjective convex processes}
	
	We now prove the generalized-equation criterion. In contrast with
	Section~\ref{sec:partial-p-factor}, model covering is not assumed. A closed
	convex process is embedded into the set-valued term through an explicit
	inner approximation, and its sum with the smooth triangular factor operator
	is assumed surjective. Openness of that sum supplies the covering, its
	convex fibres supply the topology required by range transfer, and an
	anisotropically negligible continuous remainder links the process to the
	original generalized equation. This is an explicit mechanism by which $F$ can supply directions missing
	from $\Psi_p(h)$.
	
	A set-valued mapping $\mathcal H:E\tto Y$ is a \emph{closed convex
		process} if $\gph\mathcal H$ is a closed convex cone.
	
	\begin{lemma}
		\label{lem:convex-process-covering}
		Let $E$ be a finite-dimensional normed space, and let $Y$ be
		finite-dimensional and equipped with the anisotropic range structure of
		Definition~\ref{def:anisotropic-structure}. Let
		$\mathcal M:E\tto Y$ be a closed convex process with
		$\rge\mathcal M=Y$.  Then there is an $\alpha>0$ such that
		$$
		\mathcal M(\ball_E)\supset\alpha\ball_Y^{\Sigma}.
		$$
		Consequently,
		$$
		(D_t\mathcal M)(\rho t\ball_E)
		\supset D_t(\alpha\rho t\ball_Y^{\Sigma})
		$$
		for all $\rho,t>0$.
	\end{lemma}
	
	\begin{proof}
		The Robinson--Ursescu theorem for closed convex processes gives linear
		openness at $(0,0)$; see
		\cite{UrsescuConvexGraph,RobinsonConvexMultivalued} and
		\cite[Theorem~3.1]{CibulkaRoubalSetValued}.
		Thus there is a $\beta>0$ such that
		$$
		\beta\mathop{\rm int}\ball_Y^\Sigma
		\subset\mathcal M(\mathop{\rm int}\ball_E).
		$$
		Choose $0<\alpha<\beta$.  Then
		$$
		\alpha\ball_Y^\Sigma
		\subset\beta\mathop{\rm int}\ball_Y^\Sigma
		\subset\mathcal M(\ball_E).
		$$
		Positive homogeneity and application of $D_t$ give the asserted
		scaling.
	\end{proof}
	
	Lemma~\ref{lem:convex-process-covering} states that this modulus is positive
	when $\mathcal M$ is surjective. Any smaller positive $\alpha$ may be used
	in the estimates below.
	
	Given a finite-dimensional subspace $E\subset X$ and a closed convex process
	$\mathcal H:E\tto Y$, define the
	\emph{set-valued triangular factor process relative to the chosen
		decomposition} $\Psi_{p,\mathcal H}^{E}(h):E\tto Y$ by
	\begin{equation*}
		\Psi_{p,\mathcal H}^{E}(h)(\xi)
		:=(\Psi_p(h)\vert_E)\xi+\mathcal H(\xi).
	\end{equation*}
	
	\subsection{The generalized-equation criterion}
	
	The theorem below separates three checks: a process-valued inner
	approximation of $F$, an anisotropically negligible continuous error, and
	control of the displacement of the base value. Surjectivity of
	$\Psi_{p,\mathcal H}^{E}(h)$ is the range condition that can combine smooth
	and set-valued directions.
	
	\begin{theorem}
		\label{thm:p-process-criterion}
		Let $p\ge2$, let $f:U\longrightarrow Y$ be a $C^p$ mapping, and fix
		$\bx\in U$ and $h\in X$ as in Section~\ref{sec:p-factor}; retain the
		definitions of $D_t$, $A_t$, and $x_t=\bx+th$ from that section. Let $Y$
		be finite-dimensional, let $E\subset X$ be a finite-dimensional subspace,
		assume that \eqref{eq:P1} and \eqref{eq:P4} hold, and let $F:X\tto Y$.
		Choose
		$\bv\in F(\bx)$ and put $\by=f(\bx)+\bv$.  Let
		$\mathcal H:E\tto Y$ be a closed convex process such that
		$$
		\rge\Psi_{p,\mathcal H}^{E}(h)=Y.
		$$
		Assume that there are $\rho_0,r_0>0$, points $v_t\in F(x_t)$, and
		continuous mappings $s_t:\rho_0t\ball_E\longrightarrow Y$ for $t\in(0,r_0]$ such
		that $x_t+\rho_0t\ball_E\subset U$ for every $t\in(0,r_0]$ and the
		following conditions hold:
		\begin{enumerate}[label={\rm (\roman*)}]
			\item
			\begin{equation}
				v_t+s_t(\xi)+(D_t\mathcal H)(\xi)
				\subset F(x_t+\xi)
				\quad\text{for every}\quad t\in(0,r_0]
				\quad\text{and}\quad\xi\in\rho_0t\ball_E;
				\label{eq:process-inner}
			\end{equation}
			\item
			\begin{equation}
				\lim_{\rho\downarrow0}\limsup_{t\downarrow0}
				\frac{1}{\rho t}
				\sup_{\xi\in\rho t\ball_E}
				\|D_t^{-1}s_t(\xi)\|_{\Sigma}=0;
				\label{eq:process-rem}
			\end{equation}
			\item
			\begin{equation}
				\frac{\|D_t^{-1}(v_t-\bv)\|_{\Sigma}}{t}\to0
				\quad\text{as}\quad t\downarrow0.
				\label{eq:process-center}
			\end{equation}
		\end{enumerate}
		Then there are $c,\rho,r>0$ such that
		\begin{equation}
			\by+D_t(ct\ball_Y^{\Sigma})
			\subset(f+F)(\bx+th+\rho t\ball_E)
			\quad\text{for every}\quad t\in(0,r].
			\label{eq:process-cover}
		\end{equation}
	\end{theorem}
	
	\begin{proof}
		Let
		$\mathcal M:=\Psi_{p,\mathcal H}^{E}(h)$ and $y_t:=f(x_t)+v_t,$
		and, for $\xi\in\rho_0t\ball_E$,
		$$
		e_t(\xi):=f(x_t+\xi)-f(x_t)-A_t\xi
		\quad\text{and}\quad
		r_t(\xi):=e_t(\xi)+s_t(\xi).
		$$
		By hypothesis, the process $\mathcal M$ is surjective. It is also closed
		and convex, since $\gph\mathcal M$ is the image of $\gph\mathcal H$ under the linear
		homeomorphism
		$$
		(\xi,z)\longmapsto(\xi,\Psi_p(h)\xi+z).
		$$
		For completeness, Lemma~\ref{lem:p-factor-scaled-error} gives the
		quantified implication: for each $\varepsilon>0$ there exist
		$\rho_\varepsilon\in (0,\rho_0]$ and $r_\varepsilon\in (0,r_0]$ such that
		$$
		\sup_{\xi\in\rho t\ball_E}
		\|D_t^{-1}e_t(\xi)\|_\Sigma
		\le\varepsilon\rho t
		$$
		whenever $0<\rho\le\rho_\varepsilon$ and
		$0<t\le r_\varepsilon$.  Indeed, the lemma gives the estimate on
		$\rho_\varepsilon t\ball_E$, and restriction to
		$\rho t\ball_E$ yields the displayed bound.  Consequently,
		$$
		\lim_{\rho\downarrow0}\limsup_{t\downarrow0}
		\frac{1}{\rho t}
		\sup_{\xi\in\rho t\ball_E}
		\|D_t^{-1}e_t(\xi)\|_\Sigma=0.
		$$
		Together with \eqref{eq:process-rem}, this gives
		$$
		\lim_{\rho\downarrow0}\limsup_{t\downarrow0}
		\frac{1}{\rho t}
		\sup_{\xi\in\rho t\ball_E}
		\|D_t^{-1}r_t(\xi)\|_{\Sigma}=0.
		$$
		
		Let $\alpha>0$ be given by
		Lemma~\ref{lem:convex-process-covering}. By the preceding double-limit
		relation, choose and fix $\rho\in(0,\rho_0]$ such that
		$$
		\limsup_{t\downarrow0}
		\frac{1}{\rho t}
		\sup_{\xi\in\rho t\ball_E}
		\|D_t^{-1}r_t(\xi)\|_{\Sigma}
		<\frac{\alpha}{8}.
		$$
		With this $\rho$ fixed, the definition of the limit superior yields an
		$r\in(0,r_0]$ such that
		$$
		\frac{1}{\rho t}
		\sup_{\xi\in\rho t\ball_E}
		\|D_t^{-1}r_t(\xi)\|_{\Sigma}
		\le\frac{\alpha}{4}
		\quad\text{for every}\quad t\in(0,r].
		$$
		Equivalently,
		$$
		\|D_t^{-1}r_t(\xi)\|_{\Sigma}
		\le\frac{\alpha\rho}{4}t
		\quad
		(\xi\in\rho t\ball_E,\ t\in(0,r]).
		$$
		Fix any $t\in(0,r]$ and let
		$\Omega_t:=\rho t\ball_E$ and $\widetilde F_t:=D_t\mathcal M$.
		The mapping $\widetilde F_t$ has a closed graph, and $r_t$ is
		continuous on $\Omega_t$.
		By Lemma~\ref{lem:convex-process-covering},
		$$
		\widetilde F_t(\Omega_t)
		\supset D_t(\alpha\rho t\ball_Y^{\Sigma}).
		$$
		The fibres of $\widetilde F_t$ over this set, restricted to
		$\Omega_t$, are nonempty, compact, and convex. Corollary~%
		\ref{thm:scaled-transfer}, on $E$ and with $Y$ endowed with
		$\|\cdot\|_\Sigma$, applied with
		$$
		S:=D_t,\quad \tau:=t,\quad A:=0,\quad F:=\widetilde F_t,\quad
		f:=r_t,\quad
		\gamma:=\frac{\alpha\rho}{2},\quad\text{and}\quad
		\delta:=\frac{\alpha\rho}{4},
		$$
		yields
		$$
		(r_t+D_t\mathcal M)(\Omega_t)
		\supset
		D_t\left(\frac{\alpha\rho}{2}t\ball_Y^{\Sigma}\right).
		$$
		On the other hand, \eqref{eq:process-inner} gives
		$$
		\begin{aligned}
			y_t+r_t(\xi)+(D_t\mathcal M)(\xi)
			&=f(x_t+\xi)+v_t+s_t(\xi)+(D_t\mathcal H)(\xi)\\
			&\subset(f+F)(x_t+\xi).
		\end{aligned}
		$$
		
		Finally, \eqref{eq:p-factor-center-small-o} and
		\eqref{eq:process-center} imply
		$$
		\frac{\|D_t^{-1}(y_t-\by)\|_{\Sigma}}{t}\to0
		\quad\text{as }t\downarrow0.
		$$
		After decreasing $r$ if necessary, this quotient is at most
		$\alpha\rho/4$.  The last two inclusions then give
		\eqref{eq:process-cover} with $c:=\alpha\rho/4$.
	\end{proof}
	
	\begin{remark}[Anisotropic blow-up interpretation]
		Condition~\eqref{eq:process-inner} has a direct rescaled form. For
		$\zeta\in\rho_0\ball_E$, define
		$$
		\widehat F_t(\zeta)
		:=
		\frac{1}{t}D_t^{-1}\bigl(F(x_t+t\zeta)-v_t\bigr)
		$$
		and
		$$
		\widehat s_t(\zeta)
		:=
		\frac{1}{t}D_t^{-1}s_t(t\zeta).
		$$
		Since $\mathcal H$ is positively homogeneous,
		condition~\eqref{eq:process-inner} is equivalent to
		$$
		\widehat s_t(\zeta)+\mathcal H(\zeta)
		\subset\widehat F_t(\zeta).
		$$
		Thus $\mathcal H$ is an inner convex-process approximation of the
		anisotropic blow-up of $F$, while condition~\eqref{eq:process-rem} makes the
		additional term negligible on shrinking correction balls.
	\end{remark}
	
	Surjectivity of the combined convex process does not by itself provide a
	single-valued correction, and the double-limit remainder condition
	\eqref{eq:process-rem} is not a Lipschitz contraction estimate.  Thus the
	preceding criterion is not a disguised application of the contractive
	alternative in Remark~\ref{rem:nadler-alternative}.
	
	\begin{remark}[Single-valued and cone-valued reductions]
		If $F(x)=\{0\}$ and $\mathcal H(\xi)=\{0\}$, then
		Theorem~\ref{thm:p-process-criterion} reduces to the single-valued
		triangular factor covering theorem relative to the chosen decomposition.
		More generally, if
		$F(x)\supset\bv+C$ locally, where
		$C_i\subset Y_i$ are closed convex cones and
		$$
		C:=\{y\in Y:P_i y\in C_i,\ i=1,\ldots,p\}
		=C_1+\cdots+C_p,
		$$
		then
		$D_tC=C$.  Taking $\mathcal H(\xi)=C$, $v_t=\bv$, and $s_t=0$
		verifies assumptions {\rm (i)}--{\rm (iii)} after decreasing $r_0$ so that
		all corresponding tubes lie in the neighborhood where
		$F(x)\supset\bv+C$. Hence the criterion
		applies whenever
		$$
		\Psi_p(h)(E)+C=Y.
		$$
		This shows that the set-valued part can supply range directions missing
		from $\Psi_p(h)$ itself.
	\end{remark}

	\subsection{Endpoint inclusions with cone constraints}
	
	Let $X$ be a Banach control space, let $\mathcal U\subset X$ be open,
	and let $\mathcal E:\mathcal U\longrightarrow Y$ be a $C^p$ endpoint
	map with finite-dimensional target. For admissible controls
	$\mathcal U_{\rm ad}\subset\mathcal U$ and a closed set $K\subset Y$,
	let
	$$
	G_{\rm end}(u):=\mathcal E(u)+K
	\quad\text{and}\quad
	\mathcal R(y):=G_{\rm end}^{-1}(y)\cap\mathcal U_{\rm ad}.
	$$
	For a base control $\bar u$ and a direction $h$, write
	$$
	\Psi_p^{\mathcal E,\bar u}(h)\xi
	:=
	\sum_{i=1}^p\frac{1}{(i-1)!}
	P_i\mathcal E^{(i)}(\bar u)[h]^{i-1}\xi.
	$$
	Standard hypotheses for controlled Carath\'eodory systems ensure the
	required smoothness of the endpoint map; see
	\cite[Section~20.3]{AgrachevSachkov} and \cite[p.~388]{TrelatEndpoint}.
	
	\begin{corollary}
		\label{cor:terminal-cone-process}
		Let $p\ge2$, equip $Y$ with the anisotropic structure of
		Definition~\ref{def:anisotropic-structure}, and let
		$\mathcal E:\mathcal U\longrightarrow Y$ be $C^p$. Fix
		$\bar u\in\mathcal U_{\rm ad}$ and set
		$$
		C:=\{y\in Y:P_i y\in C_i,\ i=1,\ldots,p\},
		\quad
		K:=C,
		\quad
		\bar v\in C,
		\quad
		\bar y:=\mathcal E(\bar u)+\bar v,
		$$
		where each $C_i\subset Y_i$ is a closed convex cone. Assume the
		triangular condition \eqref{eq:P1} with
		$f:=\mathcal E$ and $\bx:=\bar u$. Choose $h\in X$ such that
		$$
		\Psi_p^{\mathcal E,\bar u}(h)h=0.
		$$
		Suppose that there are a finite-dimensional subspace $E_0\subset X$ and
		$\rho_0,r_0>0$ such that
		$$
		\bar u+th+\rho_0t\ball_{E_0}
		\subset\mathcal U_{\rm ad}\cap\mathcal U
		\qquad (0<t\le r_0).
		$$
		If
		$$
		\Psi_p^{\mathcal E,\bar u}(h)(E_0)+C=Y,
		$$
		then there are $c,\rho,r>0$ such that
		$$
		\bar y+D_t(ct\ball_Y^\Sigma)
		\subset
		G_{\rm end}(\bar u+th+\rho t\ball_{E_0})
		\qquad (0<t\le r).
		$$
		Moreover, for some $M>0$ and a neighborhood $V$ of $\bar y$,
		$$
		\dist(\bar u,\mathcal R(y))
		\le
		M\max_{i\in\mathcal I}\|P_i(y-\bar y)\|^{1/i}
		\qquad (y\in V).
		$$
	\end{corollary}
	
	\begin{proof}
		Apply Theorem~\ref{thm:p-process-criterion} with
		$f:=\mathcal E$, $F\equiv C$, $E:=E_0$, $v_t:=\bar v$,
		$s_t:=0$, and $\mathcal H(\xi):=C$. Since $D_tC=C$, the assumed range
		condition makes $\Psi_p^{\mathcal E,\bar u}(h)\vert_{E_0}+\mathcal H$
		surjective, while all remainder and centring terms vanish. The covering
		follows from the theorem, and the admissible-tube inclusion allows
		Corollary~\ref{cor:moving-tube-inverse} to be applied with
		$\mathcal A:=\mathcal U_{\rm ad}$.
	\end{proof}
	
	\section{Examples and applications}
	\label{sec:examples}
	
	\subsection{Selected-direction regularity}
	
	The examples below have three distinct purposes: they separate selected-
	direction regularity from the standard all-directions notions, show how a
	set-valued term can complete a missing factor range and how acyclic exact-
	model fibres may go beyond convex-process reductions, and finally realize
	the anisotropic estimate sharply for the Grushin endpoint map.
	
	\begin{example}[Sharpness along a single regular direction without strong
		$p$-regularity]
		\label{ex:directional-not-strong}
		Let $p\ge2$, let $X:=\R^4$ and $Y:=\R^2$, both with their Euclidean
		norms, and define
		$$
		f(x_1,x_2,x_3,x_4)
		:=
		\bigl(x_4,x_1^{p-1}x_2+x_3^{p+1}\bigr).
		$$
		Take $\bx:=0$ and set
		$$
		Y_1:=\operatorname{span}\{(1,0)\}\quad\text{and}
		\quad
		Y_p:=\operatorname{span}\{(0,1)\}.
		$$
		When $p\ge3$, set $Y_i:=\{0\}$ for $i=2,\ldots,p-1$.
		The second component of $f$ contains no term of degree smaller than
		$p$, so the triangular factor condition holds.
		For this example, the chosen decomposition is also a standard recursive
		$p$-factor decomposition at $0$: indeed,
		$Y_1=\operatorname{im}f'(0)$, the derivatives $f^{(i)}(0)$ vanish for
		$2\le i<p$, and the projected derivative of order $p$ generates the
		remaining block $Y_p$.  Thus the comparison with standard $p$-regularity
		below is legitimate.
		
		Let $h:=e_1$. Direct differentiation gives
		$$
		P_pf^{(p)}(0)[h]^{p-1}\xi
		=
		\bigl(0,(p-1)!\xi_2\bigr).
		$$
		Consequently, with the factorial normalization used in this paper, we have
		$$
		\Psi_p(h)\xi=(\xi_4,\xi_2)
		\quad \text{and}\quad
		\Psi_p(h)h=0.
		$$
		Thus $\Psi_p(h)$ is surjective and
		$$
		\Psi_p(h)\bigl(\operatorname{span}\{e_2,e_4\}\bigr)=Y.
		$$
		The scaled model is even exact on this slice: for
		$\alpha,\beta\in\R$,
		$$
		f(te_1+\alpha e_2+\beta e_4)
		=
		(\beta,t^{p-1}\alpha)
		=D_t(\beta,\alpha).
		$$
		Thus $f$ is $p$-regular along $h$ in the standard terminology, and the
		selected-direction hypothesis of
		Theorem~\ref{thm:p-factor-covering} is satisfied.
		
		Since the chosen decomposition is the standard recursive one, comparison
		with the usual all-directions notions is legitimate. Consider $k:=e_3$.
		Then
		$$
		P_i f^{(i)}(0)[k]^i=0,
		\qquad i=1,\ldots,p,
		$$
		so $k$ belongs to the exact $p$-kernel. On the other hand, the term
		$x_3^{p+1}$ is invisible to the order-$p$ factor operator at the origin,
		and
		$$
		\Psi_p(k)\xi=(\xi_4,0).
		$$
		Thus $\Psi_p(k)$ is not surjective. Therefore $f$ is neither
		$p$-regular at $0$ in the standard all-directions sense nor strongly
		$p$-regular there; see \cite[Definitions~6 and~7]{BednarczukEtAlSurvey}.
		In contrast, $\Psi_p(h)$ is surjective, so
		Theorem~\ref{thm:p-factor-covering} applies along the selected direction
		$h=e_1$ and yields fixed-base H\"older semiregularity of order $1/p$.
		
		This exponent is sharp. For $0<|s|\le1$, set $y^s:=(0,s)$,
		$a:=|s|^{1/p}$, and $\sigma:=\operatorname{sgn}(s)$. Then
		$$
		x^s:=(a,\sigma a,0,0)\in f^{-1}(y^s),
		\qquad
		\|x^s\|=\sqrt2\,|s|^{1/p}.
		$$
		Conversely, if $x\in f^{-1}(y^s)$ and $r:=\|x\|$, then
		$$
		|s|
		\le |x_1|^{p-1}|x_2|+|x_3|^{p+1}
		\le r^p+r^{p+1}.
		$$
		If $r\le1$, this gives $r\ge2^{-1/p}|s|^{1/p}$; if $r>1$, the same
		lower bound is immediate because $|s|\le1$. Therefore
		$$
		2^{-1/p}|s|^{1/p}
		\le \dist\bigl(0,f^{-1}(y^s)\bigr)
		\le \sqrt2\,|s|^{1/p},
		$$
		which rules out every H\"older semiregularity order $\alpha>1/p$.
	\end{example}

	\subsection{Set-valued range completion and acyclic fibres}
	
	The first example in this subsection realizes the missing-direction
	mechanism of Theorem~\ref{thm:p-process-criterion}. The second shows that
	the exact partial model may have contractible nonconvex fibres even when
	no compatible zero-remainder convex-process inner model is available on
	the chosen correction space.
	
	\begin{example}[A nonconstant set-valued term supplies an essential
		missing direction]
		\label{ex:set-valued-missing-directions}
		For $r\in\R$, write $r_+:=\max\{r,0\}$. Let
		$$
		X:=\R^3,\quad Y:=\R^2,\quad\text{and}\quad U:=X,
		$$
		with their Euclidean norms, and define
		$$
		f(x_1,x_2,x_3):=(x_1x_2,0)
		$$
		and
		\begin{equation}
			F(x_1,x_2,x_3)
			:=
			\{0\}\times
			[-x_{1,+}x_{3,+},\,x_{1,+}x_{3,+}].
			\label{eq:nonconstant-F}
		\end{equation}
		The mapping $F$ is nonconstant, has nonempty compact convex values,
		and has a closed graph, because
		$$
		\gph F
		=
		\{(x,y)\in\R^3\times\R^2:
		y_1=0,\ |y_2|\le x_{1,+}x_{3,+}\}.
		$$
		
		Take $p=2$, $\bx:=0$, $\bv:=0$, $h:=e_1$, and
		$$
		Y_1:=\{0\},\quad Y_2:=Y.
		$$
		The initial zero block is permitted by
		Remark~\ref{rem:grading-dependence}; it records that no target direction
		is assigned first-order scale.
		Thus $P_1=0$, $P_2=I$, and $D_t=tI$. Moreover,
		$$
		\Psi_2(h)\xi=(\xi_2,0),
		\qquad
		\Psi_2(h)h=0.
		$$
		Consequently, the smooth triangular factor operator $\Psi_2(h)$ relative to
		the chosen decomposition has range
		$\R\times\{0\}$ and misses the vertical direction.
		
		Let $E:=\operatorname{span}\{e_2,e_3\}$ and define
		$\mathcal H:E\tto Y$ by
		$$
		\mathcal H(0,a,b)
		:=
		\begin{cases}
			\{0\}\times[-b,b],&b\ge0,\\
			\varnothing,&b<0.
		\end{cases}
		$$
		Its graph is
		$$
		\gph\mathcal H
		=
		\{((0,a,b),(0,z)):b\ge0,\ |z|\le b\},
		$$
		which is a closed convex cone. Hence $\mathcal H$ is a closed convex
		process. Moreover,
		$$
		\Psi_{2,\mathcal H}^{E}(h)(0,a,b)
		=
		\begin{cases}
			\{(a,z):|z|\le b\},&b\ge0,\\
			\varnothing,&b<0.
		\end{cases}
		$$
		This process is surjective: for every $(u,v)\in\R^2$,
		$$
		(u,v)\in
		\Psi_{2,\mathcal H}^{E}(h)(0,u,|v|).
		$$
		
		Let
		$$
		x_t:=(t,0,0),\quad v_t:=0,\quad\text{and}\quad s_t:=0.
		$$
		Then $v_t\in F(x_t)$. If $\xi=(0,a,b)\in E$ and $b\ge0$, then
		$$
		(D_t\mathcal H)(\xi)
		=
		\{0\}\times[-tb,tb]
		=F(x_t+\xi),
		$$
		whereas for $b<0$ the left-hand side is empty. Thus the
		inner-approximation condition of
		Theorem~\ref{thm:p-process-criterion} holds for every $\xi\in E$.
		The remainder and centring conditions hold with value zero. In fact,
		the smooth approximation is also exact on $x_t+E$:
		$$
		f(x_t+\xi)-f(x_t)-A_t\xi
		=(ta,0)-t(a,0)=0.
		$$
		Therefore all assumptions of
		Theorem~\ref{thm:p-process-criterion} are satisfied, although
		$\Psi_2(h)$ itself is not surjective.
		
		Theorem~\ref{thm:p-process-criterion} therefore yields the corresponding
		quadratic-scale covering and square-root fixed-base inverse estimate.
		The set-valued character is essential at the level of the
		decomposition-relative process model. Any linear selection $L$ of
		$\mathcal H$ must have the form
		$$
		L(0,a,b)=(0,\lambda b)
		$$
		with $\lambda\in[-1,1]$. Hence
		$$
		\bigl(\Psi_2(h)+L\bigr)(\domm\mathcal H)
		=
		\mathbb R\times\lambda[0,+\infty)
		\ne\mathbb R^2.
		$$
		Thus no linear single-valued selection supplies both signs of the
		missing vertical direction; the whole interval-valued process is needed.
		Moreover,
		$$
		\operatorname{diam}F(x)=2x_{1,+}x_{3,+},
		$$
		so $F$ is genuinely nonconstant and is not locally a translation of a
		fixed set.
	\end{example}

	\begin{example}[A hybrid inclusion with contractible nonconvex fibres]
		\label{ex:hybrid-contractible-fibres}
		Let $X:=\mathbb R^6$, equipped with the maximum norm, and write
		$$
		x=(g,a,b_+,c_+,b_-,c_-).
		$$
		Consider the closed mode cones
		$$
		\begin{aligned}
			\mathcal D_+
			&:=\{x:g\ge0,\ b_+\ge0,\ b_-=c_-=0\},\\
			\mathcal D_-
			&:=\{x:g\ge0,\ b_-\ge0,\ b_+=c_+=0\}.
		\end{aligned}
		$$
		For $x\in\mathcal D_\pm$, set
		$$
		B_\pm(x):=g(b_\pm+|c_\pm|)
		$$
		and define
		$$
		\begin{aligned}
			F_+(x)&:=
			\begin{cases}
				\{0\}\times[0,B_+(x)],&x\in\mathcal D_+,\\
				\varnothing,&x\notin\mathcal D_+,
			\end{cases}\\
			F_-(x)&:=
			\begin{cases}
				\{0\}\times[-B_-(x),0],&x\in\mathcal D_-,\\
				\varnothing,&x\notin\mathcal D_-.
			\end{cases}
		\end{aligned}
		$$
		Finally, let
		$$
		f(x):=(a,0),
		\qquad
		F(x):=F_+(x)\cup F_-(x),
		\qquad
		G:=f+F.
		$$
		The two modes meet only on the inactive set
		$b_+=c_+=b_-=c_-=0$, where both values reduce to $\{0\}$. Hence $F$
		has a closed graph and $F(0)=\{0\}$.
		
		Take $p=2$,
		$$
		Y_1:=\mathbb R\times\{0\},
		\qquad
		Y_2:=\{0\}\times\mathbb R,
		\qquad
		D_t(y_1,y_2):=(y_1,ty_2),
		$$
		and let $\bx=\by=\bv=0$ and $h=(1,0,0,0,0,0)$. Use the correction
		space
		$$
		E:=\{(0,\alpha,\beta_+,\gamma_+,\beta_-,\gamma_-):
		\alpha,\beta_\pm,\gamma_\pm\in\mathbb R\}.
		$$
		Then
		$$
		\Psi_2(h)\xi=(\alpha,0),
		\qquad
		\Psi_2(h)h=0.
		$$
		Moreover, $x_t=th$, $v_t=0\in F(x_t)$, $y_t=0$, and the smooth model is
		exact. On the plus mode,
		$$
		\mathcal G_t^2(x_t+\xi)
		=
		\{\alpha\}\times[0,t(\beta_++|\gamma_+|)],
		$$
		while on the minus mode the corresponding interval is
		$$
		\{\alpha\}\times[-t(\beta_-+|\gamma_-|),0].
		$$
		
		Fix $\rho,t>0$ and put $R:=\rho t$. If
		$$
		w=(\alpha,t\eta)\in D_t(R\ball_Y^\Sigma),
		$$
		then $|\alpha|+|\eta|\le R$. For $\eta\ge0$ choose
		$$
		\xi=(0,\alpha,\eta,0,0,0),
		$$
		and for $\eta<0$ choose
		$$
		\xi=(0,\alpha,0,0,-\eta,0).
		$$
		In either case $\|\xi\|\le R$ and
		$w\in\mathcal G_t^2(x_t+\xi)$. Therefore
		$$
		D_t(\rho t\ball_Y^\Sigma)
		\subset
		\mathcal G_t^2(x_t+\rho t\ball_E).
		$$
		
		The localized fibres are contractible but generally nonconvex. For
		$\eta\ne0$, apart from fixed and inactive coordinates, the fibre equals
		$$
		S_{\lambda,R}
		:=
		\{(\beta,\gamma)\in[0,R]\times[-R,R]:
		\beta+|\gamma|\ge\lambda\},
		\qquad
		\lambda:=|\eta|.
		$$
		This set is star-shaped with respect to $(R,0)$, because along the segment
		to $(R,0)$ one has
		$$
		\beta_s+|\gamma_s|
		=
		(1-s)(\beta+|\gamma|)+sR
		\ge\lambda.
		$$
		It is nonconvex when $\lambda>0$, since $(0,\lambda)$ and
		$(0,-\lambda)$ belong to it but their midpoint does not. For $\eta=0$,
		the two mode fibres are star-shaped with respect to their common inactive
		point, so their union is star-shaped as well. Thus every relevant fibre
		is compact and contractible. Definition~\ref{def:exact-partial-transfer}
		is satisfied with $a=1$, and Theorem~\ref{thm:exact-partial-transfer}
		applies.
		
		This example cannot be reduced to a compatible zero-remainder exact
		convex-process inner model on $E$. Indeed, suppose that a closed convex
		process $\mathcal H:E\tto\mathbb R^2$ satisfies
		$$
		D_t\mathcal H(\xi)\subset F(x_t+\xi)
		$$
		for all sufficiently small $t>0$ and all
		$\xi\in\rho_0t\ball_E$. Positive homogeneity and rescaling imply
		$$
		\domm\mathcal H\subset C_+\cup C_-,
		$$
		where
		$$
		\begin{aligned}
			C_+&:=\{\xi\in E:\beta_+\ge0,\ \beta_-=\gamma_-=0\},\\
			C_-&:=\{\xi\in E:\beta_-\ge0,\ \beta_+=\gamma_+=0\}.
		\end{aligned}
		$$
		Since the domain of a convex process is a convex cone, it must be
		contained entirely in $C_+$ or entirely in $C_-$. Indeed, if it contained
		$\xi_+\in C_+\setminus C_-$ and $\xi_-\in C_-\setminus C_+$, then
		$\xi_++\xi_-$ would belong to the domain but not to $C_+\cup C_-$.
		For $\xi\in\domm\mathcal H$ and $z\in\mathcal H(\xi)$, apply the
		inclusion to a sufficiently small positive multiple of $\xi$ and use
		positive homogeneity. It follows that all values of $\mathcal H$ have,
		respectively, nonnegative or nonpositive vertical component. On
		$C_+\cap C_-$, the right-hand side of the inner inclusion is $\{0\}$;
		since $D_t$ is injective, $\mathcal H(\xi)=\{0\}$ for every
		$\xi\in\domm\mathcal H\cap C_+\cap C_-$. Consequently,
		$\Psi_2(h)+\mathcal H$ misses one open vertical half-axis and is not
		surjective. Hence the exact-model transfer applies beyond the
		zero-remainder convex-process reduction.
	\end{example}
	
	\subsection{The Grushin endpoint map}
	
	The final example has a different role from the preceding missing-direction
	and hybrid constructions. Its smooth triangular factor operator is already
	surjective; the purpose is to realize the anisotropic scaling for a classical
	endpoint map and to prove that the resulting square-root law is sharp.
	
	\begin{example}[A bilinear terminal-control problem]
		\label{ex:bilinear-terminal-control}
		Consider the classical Grushin system
		$$
		\dot z_1(s)=u_1(s),\qquad
		\dot z_2(s)=z_1(s)u_2(s),\qquad z(0)=0.
		$$
		Let $\mathcal U:=X:=L^\infty(0,1;\R^2)$, with the pointwise Euclidean
		norm convention
		$$
		\begin{aligned}
			\|u\|_X
			&:=\mathop{\rm ess\,sup}_{s\in(0,1)}
			\sqrt{|u_1(s)|^2+|u_2(s)|^2},\\
			\mathcal U_{\rm ad}
			&:=\{u\in X:|u_i(s)|\le1
			\text{ for almost every }s,\ i=1,2\}.
		\end{aligned}
		$$
		The endpoint map is the continuous quadratic polynomial
		$$
		\mathcal E(u):=\left(
		\int_0^1u_1(s)\,ds,
		\int_0^1\!\left(\int_0^su_1(\sigma)\,d\sigma\right)u_2(s)\,ds
		\right).
		$$
		Thus $\mathcal E$ is $C^\infty$.  With
		$K:=\{0\}\times\R_+$, the inclusion
		$y\in\mathcal E(u)+K$ means
		$z_1(1)=y_1$ and $z_2(1)\le y_2$.
		
		Take $\bar u=\bar y=\bar v=0$,
		$Y_1=\R\times\{0\}$, and $Y_2=\{0\}\times\R$.  Define
		$$
		\begin{aligned}
			h(s):=\bigl(\sqrt{2\pi}\cos(\pi s),0\bigr),\quad
			e_1(s):=(1,0),\quad
			e_2(s):=\bigl(0,\sqrt{2\pi}\sin(\pi s)\bigr),\quad
			E_0:=\operatorname{span}\{e_1,e_2\}.
		\end{aligned}
		$$
		Since $P_2\mathcal E'(0)=0$, the triangular condition holds, and
		$$
		\Psi_2^{\mathcal E,\bar u}(h)v:=
		\left(
		\int_0^1v_1(s)\,ds,
		\sqrt{\frac{2}{\pi}}
		\int_0^1\sin(\pi s)v_2(s)\,ds
		\right).
		$$
		Moreover,
		$$
		\Psi_2^{\mathcal E,\bar u}(h)h=0,\quad
		\Psi_2^{\mathcal E,\bar u}(h)e_1=(1,0),
		\quad \Psi_2^{\mathcal E,\bar u}(h)e_2=(0,1).
		$$
		Thus $\Psi_2^{\mathcal E,\bar u}(h)(E_0)=\R^2$. For sufficiently small
		$\rho_0,r_0>0$, the correction tube lies in
		$\mathcal U_{\rm ad}$, so Corollary~\ref{cor:terminal-cone-process}
		gives $M>0$ and a neighborhood $V$ of $0$ such that
		$$
		\dist(0,\mathcal R(y))
		\le M\max\{|y_1|,|y_2|^{1/2}\}
		\quad (y\in V).
		$$
		
		Relative to the chosen $L^\infty$ norm with pointwise Euclidean norm,
		the exponent and its vertical constant are sharp. For
		$\eta\in(0,1/(2\pi)]$, let
		$$
		y^\eta:=(0,-\eta)
		\quad\text{and}\quad
		u^\eta(s):=\sqrt{2\pi\eta}
		\bigl(\cos(\pi s),-\sin(\pi s)\bigr).
		$$
		Then $u^\eta\in\mathcal U_{\rm ad}$,
		$\mathcal E(u^\eta)=y^\eta$, and
		$\|u^\eta\|_X=\sqrt{2\pi\eta}$.
		Conversely, if $u\in\mathcal R(y^\eta)$, then
		$z_1(0)=z_1(1)=0$ and $z_2(1)\le-\eta$.  Hence, writing
		$\|\cdot\|_2$ for the $L^2(0,1)$-norm, Cauchy--Schwarz and
		Wirtinger's inequality give
		$$
		\begin{aligned}
			\eta&\le
			\left|\int_0^1z_1(s)u_2(s)\,ds\right|
			\le\|z_1\|_2\|u_2\|_2\\
			&\le
			\frac1\pi\|u_1\|_2\|u_2\|_2
			\le
			\frac1{2\pi}\bigl(\|u_1\|_2^2+\|u_2\|_2^2\bigr)
			\le
			\frac1{2\pi}\|u\|_X^2,
		\end{aligned}
		$$
		and therefore
		$\dist(0,\mathcal R(y^\eta))=\sqrt{2\pi\eta}$.
		This agrees with the explicit Grushin geodesics in
		\cite[Section~9]{AgrachevLeeGrushin}; for $\eta=t^2$,
		$u^{t^2}=th-te_2$, consistently with $D_t=P_1+tP_2$.
	\end{example}

	\section{Conclusions and outlook}
	\label{sec:conclusion}
	
	We established a selected-direction triangular covering theorem relative
	to a fixed indexed target decomposition and derived the corresponding
	componentwise inverse estimate.  Initial and intermediate zero blocks are
	allowed and encode missing effective orders, while $p$ denotes the
	largest active block index.  Each nonzero target block $Y_i$ is controlled
	with exponent $1/i$.  The componentwise estimate is
	decomposition-relative, and the resulting scalar H\"older exponent $1/p$
	cannot be improved in general, as shown by
	Example~\ref{ex:directional-not-strong}.
	
	The scaled acyclic range-transfer statement is a direct consequence of
	the Eilenberg--Montgomery fixed-point theorem and provides a convenient
	way to separate model covering, approximation error, and correction-fibre
	topology. It is applied separately at each scale; no uniformity with
	respect to the scale parameter is asserted. For generalized equations,
	the exact partial model is used first only as an auxiliary transfer device:
	its covering and acyclic-fibre properties are assumed and then passed to
	the original generalized equation. The principal structural result is the
	convex-process criterion, in which an explicit inner approximation can be
	checked and can generate target directions missing from the smooth
	triangular factor operator.
	
	The examples distinguish the selected-direction condition from strong
	$p$-regularity, show through Remark~\ref{rem:logarithmic-remainder} that the $C^p$
	triangular hypotheses need not imply the strict weighted power-gap
	assumptions of the cited $\lambda$-truncation theorem, and illustrate
	nonconvex contractible correction fibres in a hybrid inclusion. They also realize an
	essential nonconstant set-valued missing-direction mechanism. The
	terminal-control application recovers the exact square-root law and
	constant along negative vertical perturbations for the bilinear Grushin
	system \cite{AgrachevLeeGrushin}.
	
	All conclusions remain fixed-base semiregularity statements. They use a
	finite-dimensional target and, in the process criterion, a
	finite-dimensional correction space; they do not assert metric
	regularity around the graph point. Natural extensions concern
	infinite-dimensional targets, weaker compactness or admissibility
	assumptions on correction fibres, and parameter-uniform
	generalized-equation criteria.

\end{document}